\documentclass{amsart}
\usepackage{graphicx}
\usepackage{graphicx,esint}
\usepackage{amssymb,amscd,amsthm,amsxtra}
\usepackage{latexsym}
\usepackage{epsfig}

\newtheorem{thm}{Theorem}[section]
\newtheorem{cor}[thm]{Corollary}
\newtheorem{lem}[thm]{Lemma}
\newtheorem{prop}[thm]{Proposition}
\theoremstyle{definition}
\newtheorem{defn}[thm]{Definition}
\theoremstyle{remark}
\newtheorem{rem}[thm]{Remark}
\numberwithin{equation}{section}

\newcommand{\R}{\mathbb R}
\newcommand{\be}{\begin{equation}}
\newcommand{\ee}{\end{equation}}

\newcommand{\ep}{\eps}

\newcommand{\eps}{\varepsilon}

\newcommand{\p}{\partial}

\newcommand{\comment}[1]{}

\begin{document}
\title[A two-phase fully nonlinear non-homogeneous problem]{ Free boundary regularity for  fully nonlinear non-homogeneous  two-phase
problems}
\author{Daniela De Silva}
\address{Department of Mathematics, Barnard College, Columbia University,
New York, NY 10027}
\email{\texttt{desilva@math.columbia.edu}}
\author{Fausto Ferrari}
\address{Dipartimento di Matematica dell' Universit\`a, Piazza di Porta S.
Donato, 5, 40126 Bologna, Italy.}
\email{\texttt{fausto.ferrari@unibo.it}}
\author{Sandro Salsa}
\address{Dipartimento di Matematica del Politecnico, Piazza Leonardo da
Vinci, 32, 20133 Milano, Italy.}
\email{\texttt{sandro.salsa@polimi.it }}

\thanks{ D.~D.~ and F.~ F.~  are supported by the ERC starting grant project 2011 EPSILON (Elliptic PDEs and Symmetry of Interfaces and Layers for Odd Nonlinearities). F.~F.~is supported by Miur  Grant (Prin): Equazioni di diffusione in ambiti sub-riemanniani e problemi geometrici associati.
S.~S.~ is supported by Miur Grant, Geometric Properties of Nonlinear Diffusion Problems.  F.~F.~\ wishes to thank the Department of Mathematics of Columbia University, New York, for the  kind hospitality.}

\begin{abstract}
We prove that flat or Lipschitz free boundaries of  two-phase free boundary
problems governed by fully nonlinear uniformly elliptic operators and with  non-zero right hand side are $C^{1,\gamma}$.
\end{abstract}

\maketitle


\section{Introduction and main results}

In this paper we continue the development of the regularity theory for free
boundary problems with forcing term, started in \cite{D} and \cite{DFS}. We
will focus on the following problem
\begin{equation}
\left\{
\begin{array}{ll}
\mathcal{F}(D^{2}u)=f, & \hbox{in $\Omega^+(u) \cup \Omega^-(u),$} \\
\  &  \\
(u_{\nu }^{+})^{2}-(u_{\nu }^{-})^{2}=1, &
\hbox{on $F(u):= \partial
\Omega^+(u) \cap \Omega.$} \\
&
\end{array}
\right.  \label{fb}
\end{equation}%
Here $\Omega \subset \mathbb{R}^{n}$ is a bounded domain and
\begin{equation*}
\Omega ^{+}(u):=\{x\in \Omega :u(x)>0\},\quad \Omega ^{-}(u):=\{x\in \Omega
:u(x)\leq 0\}^{\circ },
\end{equation*}%
while $u_{\nu }^{+}$ and $u_{\nu }^{-}$ denote the normal derivatives in the
inward direction to $\Omega ^{+}(u)$ and $\Omega ^{-}(u)$ respectively. Also, $f\in L^{\infty }(\Omega )$ is continuous in $\Omega^+(u) \cup \Omega^-(u)$. $F(u)$ is called the \emph{free boundary}.

 $\mathcal{F}$ is a fully nonlinear uniformly elliptic operator,
 that is there exist $0<\lambda\leq \Lambda$   positive constants such that for every  $M, N\in \mathcal{S}^{n\times n},$  with  $N\geq 0,$
$$
\lambda\| N\|\leq \mathcal{F}(M+N)-\mathcal{F}(M)\leq \Lambda \| N\|,
$$
where $\mathcal{S}^{n\times n}$ denotes the set of real  $n\times n$ symmetric matrices. We write $N \geq 0,$ whenever $N$ is non-negative definite. Also, $\|M\|$ denotes the $(L^2,L^2)$-norm of $M$, that is $\|M\|= \sup_{|x|=1} |Mx|$. Finally, we assume that  $\mathcal{F}(0)=0.$


When $f\equiv 0$ and $\mathcal{F}$ is homogeneous of degree one, several
authors extended the results of the seminal works of Caffarelli (\cite{C1,C2}%
) to various kind of nonlinear operators. Wang (\cite{W1,W2}) considered $%
\mathcal{F}=\mathcal{F}\left( D^{2}u\right) $ concave, Feldman (\cite{F1})
enlarged the class of operators to $\mathcal{F}=\mathcal{F}\left(
D^{2}u,Du\right) $ without concavity assumptions, Ferrari and Argiolas (\cite{Fe1,AF}) added $x$-dependence in $\mathcal{F}$, with $\mathcal{F}\left(
0,0,x\right) \equiv 0$.

All these papers follows the general strategy developed in \cite{C1,C2},
that however seems not so suitable when distributed
sources are present.

In \cite{D}, De Silva introduced a new technique to prove the smoothness of the free boundary for
 one-phase problems governed by non-homogeneous linear elliptic equations. As we show in
\cite{DFS}, her method is flexible enough to deal with general two-phase
problems for linear operators. Here we enforce the same technique to prove
regularity of flat free boundaries for problem (\ref{fb}). Our main result is
the following, where we denote with $B_r$ the ball of radius $r$ centered at 0 (for the definition of viscosity solution see Section \ref{section2}).

\begin{thm}[Flatness implies $C^{1,\protect\gamma }$]
\label{flatmain1} Let $u$ be a Lipschitz viscosity solution to %
\eqref{fb} in $B_{1}$. Assume that $f\in L^{\infty }(B_{1})$ is continuous {%
in $B_{1}^{+}(u)\cup B_{1}^{-}(u).$} There exists a constant $\bar{\delta}%
>0, $ depending only on $n,\lambda ,\Lambda ,\Vert f\Vert _{\infty }$ and $%
Lip(u) $ such that, if
\begin{equation}
\{x_{n}\leq -\delta \}\subset B_{1}\cap \{u^{+}(x)=0\}\subset \{x_{n}\leq
\delta \},  \label{flatness}
\end{equation}%
with $0\leq \delta \leq \bar{\delta},$ then $F(u)$ is $C^{1,\gamma }$ in $%
B_{1/2}$.
\end{thm}

Expressely note that we assume for $\mathcal{F}$ neither concavity nor
homogeneity of degree one.

When $\mathcal{F}$ is homogeneous of degree one (or when $\mathcal{F}_{r}(M)$
has a limit $\mathcal{F}^{\ast }(M),$ as $r\rightarrow 0,$ which is homogeneous of degree one) we
can also prove the following \emph{Lipschitz implies smoothness} result.

\begin{thm}[Lipschitz implies $C^{1,\protect\gamma }$]
\label{Lipmain} Let $\mathcal{F}$ be homogeneous of degree one and $u$ be a
Lipschitz viscosity solution to \eqref{fb} in $B_{1}$, with $0\in F(u)$.
Assume that $f\in L^{\infty }(B_{1})$ is continuous {in $B_{1}^{+}(u)\cup
B_{1}^{-}(u).$} If $F(u)$ is a Lipschitz graph in a neighborhood of $0$,
then $F(u)$ is $C^{1,\gamma }$ in a (smaller) neighborhood of $0$.
\end{thm}

Theorem \ref{Lipmain} follows from Theorem \ref{flatmain1} and the main
result in \cite{F1}, via a blow-up argument.\smallskip

\smallskip
As we have already mentioned, to prove Theorem \ref{flatmain1} we will use the technique introduced in \cite
{D,DFS}. In particular, the structure of our paper parallel the one in
\cite{DFS}. Thus, the proof of Theorem \ref{flatmain1} is obtained through
an iterative improvement of flatness via a suitable compactness and
linearization argument. A crucial tool is the $C^{1,\alpha }$ regularity of
the solution of the linearized problem, that turns out to be a transmission
problem in the unit ball, governed by two different fully nonlinear
operators in two half-balls. Section 3 is devoted to prove this
regularity result, that, we believe, could be interesting in itself.

As it is common in two-phase free boundary problems, the main difficulty in the
analysis comes from the case when $u^{-}$ is degenerate, that is very close
to zero without being identically zero. In this case the flatness assumption
does not guarantee closeness of $u$ to an \textquotedblleft optimal"
(two-plane) configuration. Thus one needs to work only with the positive
phase $u^{+}$ to balance the situation in which $u^{+}$ highly predominates
over $u^{-}$ and the case in which $u^{-}$ is not too small with respect to $%
u^{+}.$ For this reason, throughout the paper we distinguish two cases, which we refer to as the non-degenerate and the degenerate case.

The paper is organized as follows. In Section 2, we provide basic definitions and reduce our main flatness theorem to a proper ``normalized" situation (i.e. closeness to a two-plane solution).  As already mentioned above, Section 3 is devoted to the linearized problem. In Section 4 we obtain the necessary Harnack inequalities which rigorously allow the
linearization of the problem. Section 5 provides the proof of the improvement of flatness lemmas. Finally, the main theorems are proved in the last section.

A remark on further generalization is in order. We have choosen the particular free boundary
condition in problem \eqref{fb} in order to better emphasize the ideas
involved in our proofs. Also, to avoid the machinery of $L^{p}$-viscosity
solution, we assume that $f$ is bounded in $\Omega $ and continuous {in $%
\Omega ^{+}(u)\cup \Omega ^{-}(u)$} but everything works with $f$ merely
bounded, measurable.

Following the lines of Sections 7-9 in \cite{DFS}, our results
can be extended to a more general class of operators $\mathcal{F}=\mathcal{F}%
\left( M,p\right) $, uniformly Lipschitz with respect to $p,$ $\mathcal{F}%
\left( 0,0\right) =0$, (homogeneous of degree one in both arguments for Theorem 1.2) with free boundary conditions given by%
\begin{equation*}
u_{\nu }^{+}=G(u_{\nu }^{-},x),
\end{equation*}%
where
\begin{equation*}
G:[0,\infty )\times \Omega \rightarrow (0,\infty )
\end{equation*}%
satisfies the following assumptions:

\begin{itemize}
\item[(1)] $G(\eta,\cdot)\in C^{0,\bar\gamma}(\Omega)$ uniformly in $%
\eta; \ G(\cdot,x)\in C^{1,\bar\gamma}([0,L])$ for every $x\in \Omega.$

\item[(2)] $G^{\prime}(\cdot, x)>0$ with $G(0,x)\geq\gamma_0>0$ uniformly
in $x$.

\item[(3)] There exists $N>0$ such that $\eta ^{-N}G(\eta ,x)$ is strictly
decreasing in $\eta $, uniformly in $x$.\bigskip
\end{itemize}

A last remark concerns existence. In our generality, the existence of Lipschitz viscosity solutions with
proper measure theoretical properties of the free boundary is an open problem
and it will be object of future investigations. When $f=0$, and $\mathcal{F=F%
}\left( D^{2}u\right) $ is concave, homogeneous of degree one, the existence issue has
been settled by Wang in \cite{W3}.

Other two  recent papers, namely \cite{AT}, \cite{RT}, deal with well posedness and regularity for free boundary problems governed by fully nonlinear operators. In \cite{AT} the authors perform a complete analysis of singular perturbation problems and their limiting free boundary problems. Of particular interest is the limiting free boundary condition, obtained through a homogenization of the governing operator, under suitable hypotheses such as rotational invariance and e.g. concavity. In \cite{RT}, a free boundary problem with power type singular absorption term is considered. In this interesting paper the authors establish existence, optimal regularity and non degeneracy of a minimal solution, together with fine measure theoretical properties of the free boundary. Further regularity of the free boundary seems to be a challenging problem.

\section{Preliminaries} \label{section2}

In this section, we state basic definitions and we show
 that our flatness Theorem \ref{flatmain1} follows from Theorem \ref{main_new} below. From now on, $U_\beta$ denotes the one-dimensional function, $$U_\beta(t) = \alpha t^+ - \beta t^-, \quad  \beta \geq 0, \quad \alpha = \sqrt{1 +\beta^2},$$ where $$t^+ = \max\{t,0\}, \quad t^-= -\min\{t,0\}.$$


Here and henceforth, all constants depending only on $n,\lambda ,\Lambda ,\Vert f\Vert
_{\infty }$ and $Lip(u)$ will be called universal.

\begin{thm}\label{main_new} Let $u$ be a (Lipschitz) solution to \eqref{fb} in $B_1$ with $Lip(u) \leq L$ and $\|f\|_{L^\infty} \leq L$. There exists a universal
constant $\bar \ep>0$ such that, if \be\label{initialass}\|u - U_{\beta}\|_{L^{\infty}(B_{1})} \leq \bar \eps\quad \text{for some $0 \leq \beta \leq L,$}\ee and
\begin{equation*} \{x_n \leq - \bar \ep\} \subset B_1 \cap \{u^+(x)=0\} \subset \{x_n \leq \bar \ep \},\end{equation*} and $$\|f\|_{L^\infty(B_1)} \leq \bar \ep,$$
then $F(u)$ is $C^{1,\gamma}$ in $B_{1/2}$.
\end{thm}

The rest of the paper is devoted to the proof of Theorem \ref{main_new}, following the strategy developed in \cite{DFS}.

\

We recall some standard fact about fully nonlinear uniformly elliptic operators. For a comprehensive treatment of fully nonlinear elliptic equations, we refer the reader to \cite{CC}.

From now on, the class of all uniformly elliptic operators with ellipticity constants $\lambda,\Lambda$ and such that $\mathcal F(0)=0$ will be denoted by $\mathcal{E}(\lambda, \Lambda).$

We start with the definition of the extremal Pucci operators, $\mathcal M^-_{\lambda,\Lambda}$ and $\mathcal M^+_{\lambda,\Lambda}$. Given $0<\lambda \leq \Lambda,$ we set
$$\mathcal{M}^-_{\lambda,\Lambda}(M) = \lambda \sum_{e_i >0} e_i + \Lambda \sum_{e_i <0} e_i,$$
$$\mathcal{M}^+_{\lambda,\Lambda}(M) = \Lambda \sum_{e_i >0} e_i + \lambda \sum_{e_i <0} e_i,$$ with  the $e_i=e_i(M)$ the eigenvalues of $M$.

In the rest of the paper, whenever it is obvious, the dependance of the extremal operators from $\lambda, \Lambda$ will be omitted.

We recall that if $\mathcal{F} \in \mathcal E(\lambda,\Lambda)$ then\begin{equation*}
\mathcal{M}^-_{\frac{\lambda}{n},\Lambda}(M)\leq \mathcal{F}(M)\leq \mathcal{M}^+_{\frac{\lambda}{n},\Lambda}(M),
\end{equation*} a fact which will be used very often throughout the paper.

Finally, it is readily verified  that if $\mathcal{F} \in \mathcal E(\lambda,\Lambda)$ is the rescaling operator defined by
$$
\mathcal{F}_r(M)=\frac{1}{r}\mathcal{F}(rM), \quad r>0
$$
then $\mathcal F_r$ is still an operator in our class $\mathcal E(\lambda,\Lambda).$

 \

We now introduce the definition of viscosity solution to our free boundary problem \eqref{fb}. First we recall some standard notion.

Given $u, \varphi \in C(\Omega)$, we say that $\varphi$
touches $u$ by below (resp. above) at $x_0 \in \Omega$ if $u(x_0)=
\varphi(x_0),$ and
$$u(x) \geq \varphi(x) \quad (\text{resp. $u(x) \leq
\varphi(x)$}) \quad \text{in a neighborhood $O$ of $x_0$.}$$ If
this inequality is strict in $O \setminus \{x_0\}$, we say that
$\varphi$ touches $u$ strictly by below (resp. above).

Let $\mathcal F \in \mathcal E(\lambda, \Lambda)$. If $v \in C^2(O)$, $O$ open subset in $\R^n,$ satisfies  $$\mathcal F(D^2 v) > f  \  \ \ (\text{resp}. <f)\quad \text{in $O$,}$$ with $f \in C(O),$ we call $v$ a (strict) classical subsolution (resp. supersolution) to the equation $\mathcal F(D^2 v) = f $ in $O$.

We say that $u \in C(O)$ is a viscosity solution to $$\mathcal F(D^2 v) = f  \quad \text{in $O$,}$$ if $u$ cannot be touched by  above (resp. below) by a strict classical subsolution (resp. supersolution) at an interior point $x_0 \in O.$

We now turn to the definition of viscosity solution to our free boundary problem \eqref{fb}.
\begin{defn}\label{defsub}
We say that $v \in C(\Omega)$ is a strict (comparison) subsolution (resp.
supersolution) to (\ref{fb}) in $\Omega$, if and only if $v \in C^2(\overline{\Omega^+(v) }) \cap  C^2(\overline{\Omega^-(v) })$ and the
following conditions are satisfied:
\begin{enumerate}
\item $ \mathcal F(D^2 v)  > f $ (resp. $< f $) in $\Omega^+(v) \cup \Omega^-(v)$;
\item If $x_0 \in F(v)$, then $$(v_\nu^+)^2 - (v_\nu^-)^2 >1\quad (\text{resp. $(v_\nu^+)^2 - (v_\nu^-)^2 <1$, $v_\nu^+(x_0) \neq 0$).}$$
\end{enumerate}

\end{defn}

Notice that by the implicit function theorem, according to our definition the free boundary of a comparison subsolution/supersolution is $C^2$.

\begin{defn}\label{defnhsol} Let $u$ be a continuous function in
$\Omega$. We say that $u$ is a viscosity solution to (\ref{fb}) in
$\Omega$, if the following conditions are satisfied:
\begin{enumerate}
\item $ \mathcal F(D^2 u) = f$ in $\Omega^+(u) \cup \Omega^-(u)$ in the
viscosity sense;
\item  Any (strict) comparison subsolution $v$ (resp. supersolution) cannot touch $u$ by below (resp. by above) at a point $x_0 \in F(v) $ (resp. $F(u)$.)
\end{enumerate}
\end{defn}

The next lemma shows that ``$\delta-$flat" viscosity solutions (in the sense of our main Theorem \ref{flatmain1}) enjoy non-degeneracy of the positive part $\delta$-away from the free boundary.
Precisely,
\begin{lem}\label{deltand}Let $u$ be a solution to \eqref{fb} in $B_2$ with  $Lip(u) \leq L$ and $\|f\|_{L^\infty} \leq L$. If
$$\{x_n \leq g(x') - \delta\} \subset \{u^+=0\} \subset \{x_n \leq g(x') + \delta\},$$ with $g$ a Lipschitz function, $Lip(g) \leq L, g(0)=0$,
then $$u(x) \geq c_0 (x_n- g(x')), \quad x \in \{x_n \geq g(x') + 2\delta\}\cap B_{\rho_0}, $$ for some $c_0, \rho_0 >0$ depending on $n,\lambda, \Lambda, L$ as long as  $\delta \leq c_0.$
\end{lem}

\begin{proof} The proof follows the lines of the analogous result in \cite{DFS}. For completeness we present the details.
All constants in this proof depend on $n,\lambda,\Lambda, L.$

It suffices to show that our statement holds for $ \{x_n \geq g(x') + C\delta\}$ for a possibly large constant $C$. Then one can apply Harnack inequality to obtain the full statement.

We prove the statement above at $x=de_n$ (recall that $g(0)=0$). Precisely, we want to show that $$u(de_n) \geq c_0 d, \quad d \geq C \delta.$$After rescaling (for simplicity we drop all subindices in the rescalings and remark that the rescaled operator preserves the same ellipticity constants as $\mathcal F$), we reduce to proving that $$u(e_n) \geq c_0$$ as long as $\delta \leq 1/C$, and $\|f\|_\infty$ is sufficiently small.
Let $\gamma>\max\{0, \frac{\Lambda}{\lambda}n(n-1)-1\}$ and $$w(x)= \frac{1}{2\gamma}(1-|x|^{-\gamma})$$ be defined on the closure of the annulus $B_2 \setminus \overline B_1$. Since $w(x)=w(|x|)$ is a radial function ($r=|x|$), we easily compute that in the appropriate system of coordinates,
$$D^2w = \frac{1}{2}r^{-\gamma -2} diag\{-(\gamma+1), 1, 1, \ldots, 1\}.$$ Thus,  \begin{align*} \mathcal M^+_{\frac{\lambda}{n}, \Lambda}(D^2 w)  &= \frac{1}{2}|x|^{-\gamma-2} ((n-1)\Lambda -\frac{\lambda}{n}(\gamma+1)).\end{align*}
Hence, for  $\|f\|_\infty$ small enough
$$\mathcal F(D^2 w) \leq \mathcal M^+_{\frac{\lambda}{n}, \Lambda}(D^2 w) <- \|f\|_\infty\quad \textrm{on  $B_2\setminus  \overline{B}_{1}$}.$$ Let $$w_t(x) = w(x+te_n).$$ Notice that $$|\nabla w_0| < 1\quad \textrm{on $\p B_{1}.$}$$

From our flatness assumption for $t>,0$ sufficiently large (depending on the Lipschitz constant of $g$), $w_t$ is strictly above $u$.  We decrease $t$ and let $\bar t$ be the first $t$ such that $w_t$ touches $u$ by above. Since $w_{\bar t}$ is a strict supersolution to $\mathcal F(D^2 u)=f$ in
$B_2\setminus\bar{B}_1$ the touching point $z$ can occur only on the
$\eta:=\frac{1}{2\gamma}(1-2^{-\gamma}) $ level set in the positive phase of $u$, and $|z|\leq C = C(L).$

Since $u$ is  Lipschitz continuous, $0< u (z) = \eta \leq L d(z, F(u))$, that is a full ball around $z$ of radius  $\eta/L$ is contained in the positive phase of $u$. Thus, for $\bar \delta$ small depending on $\eta, L$ we have that $B_{\eta/2L}(z) \subset \{x_n \geq  g(x') + 2 \bar \delta\}$.

Since $x_n =g(x') + 2 \bar \delta$ is Lipschitz we can connect $e_n$ and $z$ with a chain of intersecting balls included in the positive side of $u$ with radii comparable to $\eta/2L$.  The number of balls depends on $L$ . Then we can apply Harnack inequality and obtain $$u(e_n) \geq c u(z)= c_0,$$ as desired.
\end{proof}

Next, we state a compactness lemma. Since its proof is standard (see Lemma 2.5 in \cite{DFS} and Proposition 2.9 in \cite{CC}), we omit the details.

\begin{lem} \label{compact_delta}Let $u_k$ be a sequence of viscosity solutions to \eqref{fb}  with operators $\mathcal F_k \in \mathcal E(\lambda, \Lambda)$ and  right-hand-sides $f_k$ satisfying  $\|f_k\|_{L^\infty} \leq L.$ Assume  $\mathcal{F}^k\to \mathcal{F}^*$
uniformly on compact sets of matrices,
 $u_k \to u^*$ uniformly on compact sets, and $\{u_k^+=0\} \to \{(u^*)^+=0\}$ in the Hausdorff distance. Then $$-L \leq \mathcal F^*(D^2 u^*) \leq L, \quad \text{in $\Omega^+(u^*) \cup \Omega^-(u^*)$} $$ in the viscosity sense and $u^*$ satisfies the free boundary condition  $$ ({u^*_\nu}^+)^2 - ({u^*_\nu}^-)^2= 1  \quad \hbox{on $F(u^*)$}$$ in the viscosity sense of Definition $\ref{defnhsol}.$
\end{lem}

We are now ready to re-formulate our main Theorem \ref{flatmain1}. To do so, we prove the following Lemma \ref{normalize}. Then,  using this lemma, our main Theorem \ref{flatmain1} follows by rescaling from Theorem \ref{main_new}, as desired.

\begin{lem}\label{normalize} Let $u$ be a solution to \eqref{fb} in $B_1$ with  $Lip(u) \leq L$ and $\|f\|_{L^\infty} \leq L$. For any $\ep >0$ there exist $\bar \delta, \bar r >0$ depending on $\ep, n, \lambda, \Lambda$ and  $L$ such that if \begin{equation*} \{x_n \leq - \delta\} \subset B_1 \cap \{u^+(x)=0\} \subset \{x_n \leq \delta \},\end{equation*} with $0 \leq \delta \leq \bar \delta,$ then
\be\label{conclusion_beta}\|u - U_{\beta}\|_{L^{\infty}(B_{\bar r})} \leq \eps \bar r\ee for some $0 \leq \beta \leq L.$

\end{lem}

\begin{proof} Given $\eps>0$ and $\bar r$ depending on $\eps$ to be specified later,
assume by contradiction that there exist a sequence $\delta_k \to 0$ and a sequence of solutions $u_k$ to the problem \eqref{fb} with operators $\mathcal F_k \in \mathcal E(\lambda, \Lambda)$, and right-hand-sides $f_k$ such that $Lip(u_k), \|f_k\| \leq L$ and \begin{equation}\label{trap} \{x_n \leq - \delta_k\} \subset B_1 \cap \{u_k^+(x)=0\} \subset \{x_n \leq \delta_k \},\end{equation} but the $u_k$ do not satisfy the conclusion \eqref{conclusion_beta}.

Then, up to a subsequence, the $u_k$ converge uniformly on compacts to a function $u^*$, and by the uniform ellipticity, $\mathcal F_k$ converges uniformly (up to a subsequence) on compact sets of matrices. In view of \eqref{trap} and the non-degeneracy  of $u_k^+$ $2\delta_k$-away from the free boundary (Lemma \ref{deltand}), we can apply our compactness lemma and conclude that $$-L \leq \mathcal F^*(D^2 u^*) \leq L, \quad \text{in $B_{1/2} \cap \{x_n \neq 0\}$}$$ in the viscosity sense and also  \be\label{FBu*} ({u^*_n}^+)^2 - ({u^*_n}^-)^2= 1  \quad \hbox{on $F(u^*)=B_{1/2} \cap \{x_n=0\},$}\ee with $$u^* >0 \quad \text{in $B_{\rho_0} \cap \{x_n >0\}$}.$$
Thus, by the Remark \ref{rem} in Section 3, $$u^* \in C^{1,\gamma}(B_{1/2} \cap \{x_n \geq 0\}) \cap C^{1,\gamma}(B_{1/2} \cap \{x_n \leq 0\})$$ for some $\gamma=\gamma(n,\lambda,\Lambda)$ and in view of \eqref{FBu*} we have that (for any $\bar r$ small)
$$\|u^* - (\alpha x_n^+ - \beta x_n^- )\|_{L^\infty(B_{\bar r})} \leq C(n,L) \bar r^{1+\gamma}$$
with $\alpha^2=1+\beta^2.$ If $\bar r$ is chosen depending on $\eps$ so that $$ C(n,L) \bar r^{1+\gamma} \leq \frac{\eps}{2} \bar r,$$ since the $u_k$ converge uniformly to $u^*$ on $B_{1/2}$ we obtain that for all $k$ large
$$\|u_k- (\alpha x_n^+ - \beta x_n^- )\|_{L^\infty(B_{\bar r})} \leq \eps \bar r,$$ a contradiction.
\end{proof}


\

\section{The linearized problem}

Theorem \ref{main_new} follows from the regularity
properties of viscosity solutions to the following transmission problem,
\begin{equation}\label{TP}
  \begin{cases}
    \mathcal{F} ^+ (D^2\tilde u(x))=0 & \text{in $B^+_\rho$}, \\
\ \\
  \mathcal{F} ^- (D^2\tilde u(x))=0 & \text{in $B^-_\rho$}, \\
\ \\

a  (\tilde u_n)^+ - b (\tilde u_n)^-=0 & \text{on $B_\rho \cap \{x_n =0\}$},
  \end{cases}\end{equation}
where $(\tilde u_n)^{+}$ (resp. $(\tilde u_n)^-$) denotes the derivative in the $e_n$ direction of $\tilde u$ restricted to $\{x_n >0\}$ (resp. $\{x_n <0\}$.) Here $\mathcal F^\pm \in \mathcal E(\lambda,\Lambda)$, $a > 0$ and $b \geq 0.$
Finally, if $B_\rho$ is the ball of radius $\rho$ centered at zero, we denote $$B_\rho^+:= B_\rho \cap \{x_n>0\}, \quad B_\rho^-=B_\rho\cap\{x_n<0\}.$$
In what follows, we sometimes write   $\mathcal{F} ^\pm (D^2\tilde u(x))=0$ in $B^\pm_\rho$, to denote both the interior equations in \eqref{TP}.

\begin{defn}\label{def_visc}
We say that $\tilde u \in C(B_1)$ is a viscosity subsolution (resp. supersolution) to \eqref{TP} if

(i)   $ \mathcal{F}^\pm (D^2\tilde u(x)) \geq 0  \ (\text{resp. $ \leq 0$})\quad \text{in $B^\pm_\rho$},$ in the viscosity sense;

(ii)  If ($\delta>0$ small)$$\varphi \in C^2(\overline{B}_\delta^+) \cap C^2(\overline{B}_\delta^-)$$ touches $\tilde u$ by above (resp. by below) at $x_0 \in \{x_n=0\}$, then $$a \varphi_n^+(x_0) - b \varphi_n^-(x_0) \geq 0 \quad (\text{resp.$ \leq 0$}).$$\end{defn}

If $\tilde u$ is both a viscosity subsolution and supersolution to \eqref{TP}, we say that $\tilde u$ is a viscosity solution to \eqref{TP}.

Equivalently, the condition (ii) above can be replaced by the following one:

\

(ii') If  $P(x')$ denotes a quadratic polynomial in $x'$ and
$$P(x')+px_n^+ - qx_n^-$$ touches $\tilde u$ by above (resp. by below) at $x_0 \in \{x_n=0\}$, then $$a p - b q \geq 0 \quad (\text{resp.$ \leq 0$}).$$

Indeed, let $\varphi \in C^2(\overline{B}_\delta^+) \cap C^2(\overline{B}_\delta^-)$ touch $\tilde u$ say by above at $0 \in \{x_n=0\}$. Then, by Taylor's theorem we obtain that $ \varphi(0)+ D' \cdot x' + (px_n^+ - qx_n^-) + C|x|^2$ also touches $u$ by above at $0$ with $p=\varphi_n^+(0), q=\varphi_n^-(0)$.
Then, for all $\eps>0$ small, in a sufficiently small neighborhood of $0$ we get that $$\varphi(0) + D' \cdot x' +C|x'|^2  + (p+\eps) x_n^+ - (q-\eps) x_n^-$$ also touches $u$ by above at $0$ and hence by (ii')
$$a(p+\eps) - b (q-\eps) \geq 0.$$
The desired inequality follows by letting $\eps \to 0.$

\

The objective of this section is to prove the following regularity result for viscosity solutions to the linearized problem \eqref{TP}. Constants depending on $n,\lambda,\Lambda$ are called universal.

\begin{thm}\label{lineareg}Let $\tilde u$ be a solution to \eqref{TP} in $B_{1}$ such that $\|\tilde u\|_\infty \leq 1$. Then $u\in C^{1,\gamma}(\overline{B_{1/2}^+}) \cap C^{1,\gamma}(\overline{B_{1/2}^-})$ with a universal bound on the $C^{1,\gamma}$ norm. In particular, there exists a universal constant $\tilde C$ such that
\be\label{lr}|\tilde u(x) - \tilde u(0) -(\nabla_{x'}\tilde u(0)\cdot x' + \tilde px_n^+ - \tilde qx_n^-)| \leq \tilde C r^{1+\gamma}, \quad \text{in $B_r$}\ee for all $r \leq 1/4$ and with
\begin{equation}\label{idapbq=0} a \tilde p - b \tilde q=0.\end{equation}
\end{thm}

Towards proving the theorem above, we introduce the following special classes of functions, in the spirit of \cite{CC}. From now on, since the parameters $a,b$ in the transmission condition are defined up to a multiplicative constant, and the problem is invariant under reflection with respect to $\{x_n=0\}$, we can assume without loss of generality that $a=1, 0 \leq b \leq 1.$

For $0 < \lambda \leq \Lambda$, and  $0 \leq b\leq 1,$ we denote by $\underline{\mathcal{S}}_{\lambda,\Lambda}$ the class of continuous functions $u$ in $B_1$ such that $$\mathcal M^+_{\lambda,\Lambda}(D^2u) \geq 0 \quad \text{in $B_1^+ \cup B_1^-,$}$$ and $u$ satisfies the condition $$(u_n)^+ - b (u_n)^- \geq 0, \quad  \hbox{on $B_1 \cap \{x_n=0\},$}$$ in the viscosity sense of Definition \ref{def_visc} (with comparison with test function touching $u$ by above).

Analogously, we denote by $\overline{\mathcal{S}}_{\lambda,\Lambda}$ the class of continuous functions $u$ in $B_1$ such that $$\mathcal M^-_{\lambda,\Lambda}(D^2u) \leq 0 \quad \text{in $B_1^+ \cup B_1^-,$}$$ and $u$ satisfies the condition $$(u_n)^+ - b (u_n)^- \leq 0, \quad  \hbox{on $B_1 \cap \{x_n=0\},$}$$ in the viscosity sense of Definition \ref{def_visc} (with comparison with test functions touching $u$ by below).

Finally we denote by $$\mathcal{S}_{\lambda,\Lambda}:= \underline{\mathcal{S}}_{\lambda,\Lambda} \cap \overline{\mathcal{S}}_{\lambda,\Lambda}.$$

First we prove the following H\"older regularity result.

\begin{thm}\label{calpha} Let $u \in \mathcal{S}_{\lambda,\Lambda}$ with $\|u\|_\infty \leq 1.$ Then $u \in C^{\alpha}(B_{1/2})$ for some $\alpha$ universal, and with a universal bound on the $C^\alpha$ norm.
\end{thm}

The Theorem above immediately follows from the next Lemma.

\begin{lem} Let $u \in \mathcal{S}_{\lambda,\Lambda}$ with $\|u\|_\infty \leq 1.$ Assume that
\be\label{ubig} u(\frac{1}{5}e_n) >0.\ee Then, there exists a universal constant $c>0$ such that
$$u \geq -1+c \quad \text{in $B_{1/3}$}.$$
\end{lem}
\begin{proof} By Harnack inequality (see Theorem 4.3 in \cite{CC}) and assumption \eqref{ubig} we have that ($\bar x= \frac{1}{5}e_n$)$$u +1 > \tilde c \quad \text{in $B_{1/20}(\bar x)$}.$$

Let $$w=\eta( \Gamma^\gamma(|x-\bar x|)+ \delta x_n^+), \quad \Gamma^\gamma(|x-\bar x|)= |x-\bar x|^{-\gamma} - (2/3)^{-\gamma}$$ be defined in the closure of the annulus
$$D:=  B_{3/4}(\bar x) \setminus \overline{B}_{1/20}(\bar x)$$ with $\gamma > \max\{0, \frac{(n-1)\Lambda}{\lambda} -1\},$ and $\eta,\delta$ to be made precise later.
Since $\Gamma^\gamma (|x-\bar x|)$ is a radial function ($r=|x-\bar x|$), we find that in the appropriate system of coordinates,
$$D^2 w = \eta  \gamma r^{-\gamma-2} diag\{(\gamma+1), -1,\ldots, -1\}$$
then, in $D$
 \begin{align*}
\mathcal M^-_{\lambda,\Lambda}(D^2 w)
 =  \eta \gamma|x-\bar x|^{-\gamma-2}(\lambda(\gamma+1) -\Lambda(n-1))>0. \end{align*}
  Since
 $$\p_n \Gamma^\gamma |_{\{x_n=0\}}>0,$$ the transmission condition
$$ (w_n)^+ - b (w_n)^- > 0 \quad \text{on $x_n=0,$}$$ is satisfied.

Finally, notice that on $\p B_{3/4}(\bar x)$ we have that $w \leq 0$ as long as $\delta$ is chosen sufficiently small (universal). Also, we choose $\eta$ so that $$w \leq \tilde c \quad \text{on $\p B_{1/20}(\bar x),$}$$ that is $$\eta (\Gamma^\gamma(1/20) +\delta (\frac{1}{20} + \frac{1}{5})) \leq  \tilde c.$$

 Combining all the facts above we obtain that $$w \leq  u +1\quad \text{on $\p D$}$$
 and $w$ is a strict (classical) subsolution to the transmission problem in $D$. By the the definition of viscosity solution, we conclude that $$w \leq u +1\quad \text{in $D.$}$$

 Our desired statement now follows from the fact that $$w \geq c \quad \text{on $B_{1/3}$}$$
for $c$ universal.
\end{proof}


\

Now, we wish to prove the following main result.

\begin{prop}\label{regularize} Let $u$ be a subsolution to \eqref{TP} in $B_1$ and let $v$ be a supersolution to \eqref{TP} in $B_1.$ Then $$u-v \in \underline{\mathcal{S}}_{\frac{\lambda}{n},\Lambda}.$$
\end{prop}

\begin{cor}\label{diffquot}Let $u$ be a viscosity solution to \eqref{TP} then for any unit vector $e'$ in the $x'$ direction, $$\frac{u(x+\eps e') - u(x)}{\eps} \in \mathcal{S}_{\frac{\lambda}{n},\Lambda}.$$\end{cor}

In view of Theorem \ref{calpha} and the Corollary above we obtain by standard arguments (see Chapter 5 in \cite{CC}) the following result.

\begin{thm}Let $u$ be a viscosity solution to \eqref{TP} with $\|u\|_\infty \leq 1.$ Then, for some $\alpha$ universal,  $u \in C^{1,\alpha}$ in the $x'$-direction in $B_{3/4}$ with $C^{1,\alpha}$ norm bounded by a universal constant.\end{thm}

The desired estimate \eqref{lr} in Theorem \ref{lineareg} now follows from the corollary above and the (boundary) regularity theory for fully nonlinear uniformly elliptic equations (for the regularity result up to the boundary see \cite{MW} or the Appendix in \cite{MS}.)

We show below how to obtain \eqref{idapbq=0}, which concludes the proof of Theorem \ref{lineareg}.

\

\textit{Proof of \eqref{idapbq=0}.}  Let us prove that $a \tilde p - b \tilde q \leq 0.$ (The other inequality follows similarly.) Without loss of generality (after subtracting a linear function), we can assume that ($r$ small)
\be\label{stima}|\tilde u - (\tilde px_n^+ - \tilde q x_n^-)| \leq \tilde C |x|^{1+\gamma}, \quad |x| \leq r.\ee
For any $\delta >0$ small,  we define,
$$w_\delta(x) = C(-|x|^2+K x_n^2) - \delta |x_n|+ \tilde p x_n^+ - \tilde q x_n^-,$$
 where
 $$r^\gamma = \frac{1}{2}\frac{\delta}{\tilde CK}, \quad C= \tilde C r ^{\gamma -1}, $$
and $K=K(n,\lambda,\Lambda)$
is chosen large enough so that
\be\label{almost}\mathcal M^-(D^2 (-|x|^2+ K x_n^2))= \frac{\lambda}{n}2(K-1) -2\Lambda n > 0.\ee

Then, using \eqref{stima} it is easy to verify that
$$w_\delta < \tilde u, \quad \text{on $\p B_r$}.$$

Let $$m = \min_{\overline B_r}(\tilde u - w_\delta)=(\tilde u - w_\delta)(x_0), \quad x_0 \in \overline B_r.$$

In view of \eqref{almost}, the minimum cannot occur in the interior $B_r \cap \{x_n \neq 0.\}$
Also, since $(\tilde u - w_\delta)(0)=0,$ we have $m \leq 0$ and hence the minimum cannot occur on $\p B_r$.
Thus, $x_0$ occurs on $\{x_n=0\}$, then $w_\delta +m$ touches $\tilde u$ by below at $x_0$ and by definition
\be \label{quasi} a (\tilde p -\delta) - b(\tilde q + \delta) \leq 0.\ee
The conclusion follows by letting $\delta \to 0$ in \eqref{quasi}.
\qed

\

We are now left with the proof of our main Proposition \ref{regularize}. First, we remark that
in the proof of this proposition we will need a pointwise boundary regularity result of the following type (see \cite{MW}).

\begin{prop} \label{BPR}Let $u$ satisfy ($\mathcal F \in \mathcal E(\lambda,\Lambda)$)$$\mathcal F(D^2u)=f \quad \text{in $B^+_1$}, \quad u(x',0)=\varphi(x') \quad \text{on $B_1 \cap \{x_n=0\}$}$$ with $\varphi$ pointwise $C^{1,\alpha}$ at $0$ and $f \in L^\infty(\overline{B_1^+})$. Then $u$ is pointwise $C^{1,\alpha}$ at $0$,  that is there exists a linear function $L_u$ such that for all $r$ small $$|u - L_u| \leq C r^{1+\alpha}, \quad \text{in $\overline{B_r^+(0)}$}$$ with $C$ depending on $n,\lambda, \Lambda, \|f\|_\infty$ and the pointwise $C^{1,\alpha}$ bound on $\varphi$.
\end{prop}

\begin{rem}\label{rem} Clearly, from the proposition above and the interior regularity estimates it follows that regularity up to the boundary holds also for a problem with right-hand side as used in Lemma \ref{normalize}.
\end{rem}

\
 Towards the proof of Proposition \ref{regularize},  we now introduce the following regularizations. Given a continuous function $u$ in $B_1$ and an arbitrary ball $B_\rho$ with $\overline{B}_\rho \subset B_1$ we define for $\eps>0$ the upper $\eps$-envelope of $u$ in the $x'$-direction,
$$u^\eps(y',y_n) = \sup_{x \in \overline{B}_\rho \cap\{x_n =y_n\}}\{u(x',y_n)  - \frac{1}{\eps}|x'-y'|^2\}, \quad y=(y',y_n) \in B_\rho.$$

The proof of the following facts is standard (see \cite{CC}):

(1) $u^\eps \in C(B_\rho)$ and $u_\eps \to u$ uniformly in $B_\rho$ as $\eps \to 0.$

(2) $u^\eps$ is $C^{1,1}$ in the $x'$-direction by below in $B_\rho$. Thus, $u^\eps$ is pointwise second order differentiable in the $x'$-direction at almost every point in $B_\rho.$

(3) If $u$ is a viscosity subsolution to \eqref{TP} in $B_1$ and $\overline{B}_r \subset B_\rho$, then for $\eps \leq \eps_0$ ($\eps_0$ depending on $u, \rho, r$) $u^\eps$ is a viscosity subsolution to \eqref{TP} in $B_r.$ This fact follows from the obvious remark that the maximum of solutions of \eqref{TP}  is a viscosity subsolution.

Analogously we can define $u_\eps$, the lower $\eps$-envelope of $u$ in the $x'$-direction which enjoys the corresponding properties.

\

We are now ready  to prove our main proposition.

\

\textit{Proof of Proposition \ref{regularize}.} In what follows, for notational simplicity, we omit the dependence of the Pucci operators from $\lambda/n, \Lambda.$

By Theorem 5.3 in \cite{CC} we only need to show that the free boundary condition is satisfied in the viscosity sense.
Let $$\varphi(x) = P(x') + px_n^+-qx_n^-$$ touch $u-v$ by above at a point $x_0 \in \{x_n=0\},$ with $P$ a quadratic polynomial. Assume by contradiction that $$p - bq < 0.$$ Without loss of generality we can assume that $\varphi$ touches $w:=u-v$ strictly and also that $\mathcal M^+(D^2P) < 0$ (by modifying $p, q$, allowing quadratic dependence on $x_n$ and possibly restricting the neighborhood around $x_0$). Let us say that on the annulus $\overline B_{2\delta}(x_0) \setminus B_{\delta/2}(x_0)$, $\varphi -w \geq \eta >0.$
Now, since $w_\eps:=u^\eps-v_\eps$ converges to $u-v$ uniformly, and $\varphi$ touches $w$ strictly by above, for $\eps$ small enough we have that (up to adding a small constant) $\varphi$ touches $w_\eps$ at some $x_\eps$ by above and say $\varphi - w_\eps \geq \eta/2$ on $\p B_{\delta}(x_\eps)$. By property (3) above and the fact that $\mathcal M^+(D^2P) < 0$ we get from the comparison principle that $x_\eps \in \{x_n=0\}$.

Now, call
$$\psi= \varphi - w_\eps - \eta/2$$
Since $\psi \geq 0$ on $\p B_{\delta}(x_\eps)$ and
$\psi(x_\eps) < 0$ we obtain by ABP estimates (see Lemma 3.5 in \cite{CC}) that the set of points in $B_\delta(x_\eps) \cap \{x_n =0\}$ where $\psi$ admits a touching plane $l(x')$, of slope less than some arbitrarily small number, by below in the $x'$-direction is a set of positive measure. We choose the slope of $l$ small enough so that $\bar \varphi= \varphi-l -\eta/2$ is above $w_\eps$ on $\p B_{\delta}(x_\eps)$ and hence in the interior.
By property (2) above we can then conclude that $\bar\varphi$ touches $w_\eps$ by above at some point $y_\eps \in \{x_n=0\}$ where $u^\eps$ and $v_\eps$ are twice pointwise differentiable in the $x'$-direction.

Now call $\bar u^\eps$ (resp. $\bar v_\eps$) the solution to $\mathcal F^\pm(D^2 w)=0$ in $B_\delta^\pm(x_\eps)$ with $w=u^\eps$ (resp. $v_\eps$) on $\p B^\pm_\delta(x_\eps).$ Also call $\bar w_\eps = \bar u^\eps - \bar v_\eps.$ Since (by Theorem 5.3 in \cite{CC}) $\mathcal M^+(D^2\bar w_\eps) \geq0$ in $B_\delta^\pm(x_\eps)$ and $\mathcal M^+(D^2\bar \varphi) < 0$ in $B_\delta^\pm(x_\eps)$ with $\bar \varphi \geq \bar w_\eps=w_\eps$ on the boundary (recall that $l$ is below $\psi$ on $\{x_n=0\}$), we conclude that $\bar \varphi$ is above $\bar w_\eps$ also in the interior and therefore it touches it by above at $y_\eps.$

Since the boundary data is twice pointwise differentiable at $y_\eps$, thus in particular it is pointwise $C^{1,\alpha}$ we conclude by pointwise $C^{1,\alpha}$ regularity that $\bar u^\eps$ is $C^{1,\alpha}$ up to $y_\eps$ that is, there exist linear functions $L_u,L_v$ such that for all $r$ small $$|\bar u^\eps - L_u| \leq C r^{1+\alpha}, \quad \text{in $B_r^+(y_\eps)$,}$$
$$|\bar v_\eps - L_v| \leq C r^{1+\alpha}, \quad \text{in $B_r^+(y_\eps)$.}$$
Since $\bar \varphi$ touches $\bar w_\eps$ by above at $y_\eps$ we get that:
$$p \geq p^+_u-p^+_v$$ where
$$p^+_u = (\bar u^\eps)^+_n(y_\eps), \quad p^+_v = (\bar v_\eps)^+_n(y_\eps).$$
Arguing similarly in $B_r^-(y_\eps)$ we also get,
$$q\leq q^-_u-q^-_v,$$
where
$$q^-_u = (\bar u^\eps)^-_n(y_\eps), \quad q^-_v = (\bar v_\eps)^-_n(y_\eps).$$
We therefore contradict the fact that $p-bq<0$ if we show that
$$p^+_u - bq^-_u \geq 0, \quad p^+_v - bq^-_v \leq 0.$$ Since the replacement $\bar u^\eps$ (resp. $\bar v_\eps$) is still a viscosity subsolution (resp. supersolution) to the boundary condition, the inequalities above follow from the next Lemma. Thus our proof is concluded. \qed

\begin{lem} Let $u$ be a viscosity solution to  ($0 \leq b \leq 1$) \begin{equation}  \label{TPs}
\left\{
\begin{array}{ll}
\mathcal F^\pm(D^2u) = 0, & \hbox{in $B_1^\pm$} \\
\  &  \\
u^+_{n} - b u^-_{n}\geq 0, & \hbox{on $B_1 \cap \{x_n=0\}.$} \\
&
\end{array}%
\right.
\end{equation} Assume that $u$ is twice differentiable at zero in the $x'$-direction. Then, $u$ is differentiable at $0$ and
$$ u_{n}^+(0) - b u_{n}^-(0) \geq 0.$$
 \end{lem}

\begin{proof} From pointwise boundary regularity we know that
there exist  linear functions $L_u$ such that for all $r$ small $$|u - L_u| \leq C r^{1+\alpha}, \quad \text{in $\overline{B_r^+(0)}$.}$$Let us assume (by subtracting a linear function) that $L_u= d^+x_n.$

Let $w$ solve $$\mathcal F^+(D^2 w)=0 \quad \text{in $B_r^+$}, \quad w=\phi_r \quad \text{on $\p B_r^+,$}$$ where
$$\phi_r:= \begin{cases}C_1|x|^{1+\alpha} & \text{if $x \in \p B_r^+ \cap \{x_n>0\}$}\\ C_2 |x'|^2 & \text{if $x \in B_r \cap \{x_n=0\}$}\end{cases}$$ with $C_1 =2 C$ and $C_2=C_1r^{\alpha-1}.$ Therefore, by the assumption that $u$ is twice pointwise differentiable at 0 we get (for $r$ small enough)
$$u - d^+x_n \leq \phi_r  \quad \text{on $\p B_r^+.$} $$
Then, by the comparison principle, \be\label{ubound} u - d^+ x_n \leq w  \quad \text{in $B_r^+.$} \ee Let $\tilde w$ be the rescaling:
$$\tilde w(x) = \frac{1}{r^{1+\alpha}} w(rx), \quad x \in B_1^+.$$ Then $\tilde w$ solves $$\mathcal G(D^2 \tilde w)=0 \quad \text{in $B_1^+$}, \quad \tilde w = C_1 |x'|^2 \quad \text{on $B_1 \cap \{x_n=0\}$}$$ where $\mathcal G(M) = r^{1-\alpha} \mathcal F^+(r^{\alpha-1} M)$ depend on $r$ but has the same ellipticity constants as $\mathcal F^+.$

By boundary $C^{1,\alpha}$ estimates we obtain that $$\|\tilde w\|_{1,\alpha} \leq C_3 \quad \text{in $\overline{B}^\pm_{1/2}$}$$ with $C_3$ universal.
Then, in particular
$$\tilde w \leq C_1|x'|^2+C_3x_n \quad \text{in $\overline{B}^+_{1/2}$}$$ and by rescaling we conclude that $$w \leq C_1 r^{\alpha-1}|x'|^2+ C_3 r^\alpha x_n, \quad \text{in $\overline{B}^+_{r/2}$}.$$
Thus, by \eqref{ubound}
$$u \leq C_1 r^{\alpha-1}|x'|^2+ C_3 r^\alpha x_n + d^+ x_n^+, \quad \text{in $\overline{B}^+_{r/2}$},$$ where we recall that $d^+= (u_n)^+(0).$
Arguing similarly in $B_r^-$ we also obtain that
$$u \leq C_1 r^{\alpha-1}|x'|^2+ C_3 r^\alpha x_n - d^- x_n^-, \quad \text{in $\overline{B}^-_{r/2}$}.$$ where $d^- = (u_n)^-(0).$
Thus,
$$\varphi= C_1r^{\alpha-1}|x'|^2+ C_3 r^\alpha x_n +  d^+x_n^+ - d^- x_n^-$$
touches $u$ by above at zero with $p=Mr^\alpha +d^+$ and $q= -Mr^\alpha+d^-$.
We conclude that
$$Mr^\alpha + bM r^\alpha + d^+ - bd^- \geq 0 $$ for all $r$ small, from which our desired claim follows.
\end{proof}

\section{Harnack inequality}\label{section4}
In this section we prove a Harnack-type inequality for ``flat" solutions to our free boundary problem \eqref{fb}.  Our strategy follows closely the arguments in \cite{DFS}. We especially point out the main technical differences in the proofs, which mostly consist of the choice of the barriers.

Throughout this section we consider a Lipschitz solution  $u$ to \eqref{fb}  with  $Lip(u) \leq L$.
As pointed out in the introduction, we distinguish two cases, the non-degenerate and the degenerate one.

\subsection{Non-degenerate case}In this case $u$ is trapped between two translations of a ``true" two-plane solution $U_\beta$ that is with $\beta \neq 0.$

\begin{thm}[Harnack inequality]\label{HI}There exists a universal constant $\bar
\ep$,  such that if $u$ is a solution of $(\ref{fb})$ that satisfies at some point $x_0 \in B_2$

\be\label{osc} U_\beta(x_n+ a_0) \leq u(x) \leq U_\beta(x_n+ b_0) \quad
\text{in $B_r(x_0) \subset B_2,$}\ee
with $$ \|f\|_{L^\infty} \leq \ep^2 \beta, \quad 0 < \beta \leq L, \quad$$ and
$$b_0 - a_0 \leq \ep r, $$ for some $\ep \leq \bar \ep,$ then
$$ U_{\beta}(x_n+ a_1) \leq u(x) \leq U_\beta(x_n+ b_1) \quad \text{in
$B_{r/20}(x_0)$},$$ with
$$a_0 \leq a_1 \leq b_1 \leq b_0, \quad
b_1 -  a_1\leq (1-c)\ep r, $$ and $0<c<1$ universal.
\end{thm}

Let
$$\tilde u_\ep(x) = \begin{cases} \dfrac{u(x) -\alpha x_n }{\alpha\ep}  \quad \text{in $B_2^+(u) \cup F(u)$} \\ \ \\ \dfrac{u(x) -\beta x_n }{\beta\ep}  \quad \text{in $B_2^-(u)$.} \end{cases}$$

From a standard iterative argument (see \cite{DFS}), we obtain the following corollary.

\begin{cor} \label{corollary}Let $u$ be as in Theorem $\ref{HI}$  satisfying \eqref{osc} for $r=1$. Then  in $B_1(x_0)$ $\tilde
u_\ep$ has a H\"older modulus of continuity at $x_0$, outside
the ball of radius $\ep/\bar \ep,$ i.e. for all $x \in B_1(x_0)$ with $|x-x_0| \geq \ep/\bar\ep$
$$|\tilde u_\ep(x) - \tilde u_\ep (x_0)| \leq C |x-x_0|^\gamma.
$$
\end{cor}

The main tool in the proof of the Harnack inequality is the following lemma.

\begin{lem}\label{main}There exists
a universal constant $\bar \ep>0$ such that if $u$ satisfies \be\label{control}  u(x) \geq U_\beta(x), \quad \text{in $B_1$}\ee with \be\label{nond}\|f\|_{L^\infty(B_1)} \leq \ep^2 \beta, \quad 0 < \beta \leq L, \ee  and at $\bar x=\dfrac{1}{5}e_n$ \be\label{u-p>ep2}
u(\bar x) \geq U_\beta(\bar x_n + \ep), \ee then \be u(x) \geq
U_\beta(x_n+c\eps) \quad \text{in $\overline{B}_{1/2},$}\ee for some
$0<c<1$ universal.  Analogously, if $$u(x) \leq U_\beta(x) \quad \text{in $B_1$}$$ and $$ u(\bar x) \leq U_\beta(\bar x_n - \eps)$$ then $$ u(x) \leq U_\beta(x_n - c \ep) \quad
\text{in $\overline{B}_{1/2}.$}$$
\end{lem}

\begin{proof} We prove the first statement. For notational simplicity we drop the sub-index $\beta$ from $U_\beta$ and the dependence of the Pucci operators from $\lambda/n, \Lambda.$

Let $$w=c(|x-\bar x|^{-\gamma} - (3/4)^{-\gamma})$$ be defined in the closure of the annulus
$$A:=  B_{3/4}(\bar x) \setminus \overline{B}_{1/20}(\bar x)$$ where $\gamma>0$ will be fixed later on.
The constant $c$ is such that
$w$ satisfies the
boundary conditions
  $$\begin{cases}
    w =0 & \text{on $\p B_{3/4}(\bar x)$}, \\
    w=1 & \text{on $\p B_{1/20}(\bar x)$}.
  \end{cases} $$ Extend $w$ to be equal
to 1 on $B_{1/20}(\bar x).$
 We claim that, $$\mathcal F (D^2w) \geq  k(n,\lambda,\Lambda)>0, \quad \text{in $A$}.
 $$

Indeed,
since $w(x)=w(|x-\bar x|)$ is a radial function ($r=|x-\bar x|$), we find that in the appropriate system of coordinates,
$$D^2 w = c \gamma r^{-\gamma-2} diag\{(\gamma+1), -1,\ldots, -1\}.$$
Then, in $A$
 \begin{align*}
\mathcal M^-(D^2 w)
 = c \gamma|x-\bar x|^{-\gamma-2}(\frac{\lambda}{n}(\gamma+1) -\Lambda(n-1)) \geq k(n,\lambda,\Lambda)>0, \end{align*}
 as long as $\gamma > \max\{0, \frac{n(n-1)\Lambda}{\lambda} -1\}.$

Thus
 $$\mathcal{F}(D^2 w) \geq \mathcal{M}^-(D^2w) \geq k(n,\lambda,\Lambda) > 0, \quad 0 \leq w \leq 1  \quad \text{in $A$.}$$

 Having provided the appropriate barrier, the proof now proceeds as in Lemma 4.3 in \cite{DFS}. For the reader's convenience we provide the details.

Notice that since $x_n >0$ in $B_{1/10}(\bar x)$  and $u \geq U$ in $B_1$ we get
$$ B_{1/10}(\bar x) \subset B_1^+(u). $$

Thus $u-U \geq
0$ and solves $\mathcal{F}(D^2(u-U))=\mathcal{F}(D^2u)=f$  in $B_{1/10}(\bar x)$. We can apply Harnack inequality
to obtain
\be\label{HInew} u(x)
- U(x) \geq c(u(\bar x)- U(\bar x)) - C \|f\|_{L^\infty} \quad \text{in $\overline B_{1/20}(\bar x)$}. \ee
From the assumptions \eqref{nond} and  \eqref{u-p>ep2}  we conclude that (for $\ep$ small enough)
\be\label{u-p>cep} u
- U \geq \alpha c\ep - C \ep^2\beta \geq \alpha c_0\ep \quad \text{in $\overline B_{1/20}(\bar x)$}. \ee
Now set $\psi =1-w$ and

\be\label{v} v(x)= U(x_n - \ep c_0 \psi(x)), \quad x \in \overline B_{3/4}(\bar x),\ee and for $t
\geq 0,$
$$v_t(x)= U(x_n - \ep c_0 \psi(x)+t\ep), \quad x \in \overline B_{3/4}(\bar x).
$$

Then,

$$\label{v<u} v_0(x)=U(x_n - \ep c_0 \psi(x)) \leq
U(x) \leq u(x) \quad x \in \overline B_{3/4}(\bar x).$$

Let $\bar t$ be the largest $t \geq 0$ such that
$$v_{t}(x) \leq u(x) \quad \text{in $\overline B_{3/4}(\bar x)$}.$$

We want to show that $\bar t \geq c_0.$ Then we get the desired statement. Indeed,
$$u(x) \geq v_{\bar t}(x) = U(x_n - \ep c_0 \psi + \bar t \ep) \geq U(x_n + c\ep) \quad \text{in $B_{1/2}  \subset \subset B_{3/4}(\bar x)$}$$ with $c$ universal. In the last inequality we used that $\|\psi\|_{L^\infty(B_{1/2})} <1.$

Suppose $\bar t < c_0$. Then at some
$\tilde x \in \overline B_{3/4}(\bar x)$ we have $$v_{\bar t}(\tilde x) =
u(\tilde x).$$ We show that such touching point can only occur on $\overline B_{1/20}(\bar x).$
Indeed, since $w\equiv 0$ on $\p B_{3/4}(\bar x)$ from the definition of $v_t$ we get that for $\bar t < c_0$

$$v_{\bar t}(x) =U(x_n -\ep c_0\psi(x) + \bar t \ep) < U(x) \leq u(x)\quad \textrm{on  $\p B_{3/4}(\bar x)$}.$$

We now show that $\tilde x$ cannot belong to the annulus $A$.
Indeed, in $A^+(v_{\bar t})$
\begin{equation*}\label{laplacev}
\mathcal{F}(D^2v_{\bar t}(x)) =\mathcal{F}(\alpha \eps c_0 D^2 w) \geq \mathcal M^-(\alpha \eps c_0 D^2 w) \geq \alpha \eps c_0 k \geq \beta \eps^2 \geq \|f\|_\infty
\end{equation*} for $\ep$ small enough. An analogous computation holds in $A^-(v_{\bar t}).$

Finally,
$$(v_{\bar t}^+)_\nu^2 - (v_{\bar t}^-)_\nu^2 = 1 + \ep^2 c_0^2 |\nabla \psi|^2  - 2\ep c_0 \psi_n \quad \text{on $F(v_{\bar t}) \cap A$}.$$
Thus,
$$(v_{\bar t}^+)_\nu^2 - (v_{\bar t}^-)_\nu^2 > 1  \quad \text{on $F(v_{\bar t}) \cap A$}$$
as long as
$$\psi_n <0  \quad \text{on $F(v_{\bar t}) \cap A$}.$$
This can be easily verified from the formula for $\psi$ (for $\ep$ small enough.)

 Thus,
$v_{\bar t}$ is a strict subsolution to $\eqref{fb}$ in $A$ which lies below $u$, hence by the definition of viscosity solutions
$\tilde x$ cannot belong to $A.$

Therefore, $\tilde x \in \overline B_{1/20}(\bar x)$ and
$$u(\tilde x)=v_{\bar t}(\tilde x)  = U(\tilde x_n + \bar t \ep) \leq U(\tilde x) + \alpha \bar t \eps < U(\tilde x) + \alpha c_0 \ep$$ contradicting \eqref{u-p>cep}.
\end{proof}

We can now prove our Theorem \ref{HI}.

\vspace{3mm}

\textit{Proof of Theorem $\ref{HI}$.} Assume without loss of generality that $x_0=0, r=1.$ We distinguish three cases.

\vspace{1mm}

{\it Case 1.} $a_0 < - 1/5.$ In this case it follows from \ref{osc} that $B_{1/10} \subset \{u<0\}$ and
$$0 \leq v(x):= \frac{u(x) - \beta(x_n+a_0)}{\beta \eps} \leq 1.$$

We recall that the operator $\mathcal{F}_{\epsilon\beta}(M)=\frac{1}{\epsilon\beta}\mathcal{F}(\epsilon\beta M) \in \mathcal{E}(\lambda,\Lambda),$
and
$$
\mathcal{F}_{\epsilon\beta}(D^2 v)=\frac{f}{\beta\epsilon}.
$$
Moreover
 $$|\mathcal{F}_{\eps \beta}(D^2 v)| \leq \eps \quad \text{in $B_{1/10}$}$$ since $\|f\|<\beta\epsilon^2.$ The desired claim follows from standard Harnack inequality applied to the function $v$.

\vspace{1mm}

{\it Case 2.} $a_0 > 1/5.$ In this case it follows from \eqref{osc} that $B_{1/5} \subset \{u>0\}$ and
$$0 \leq v(x):= \frac{u(x) - \alpha(x_n+a_0)}{\alpha \eps} \leq 1.$$

As before,$$
\mathcal{F}_{\epsilon\alpha}(D^2 v)=\frac{1}{\epsilon\alpha}\mathcal{F}(\epsilon\alpha D^2v)= \frac{f}{\alpha\epsilon},
$$
hence
$$|\mathcal{F}_{\eps \alpha}(D^2v)| \leq \eps \quad \text{in $B_{1/5}$}$$ since $\|f\|\leq \beta \epsilon^2 \leq \alpha \eps^2.$ Again, the desired claim follows from standard Harnack inequality for $v$.

\vspace{1mm}

{\it Case 3.} $|a_0| \leq 1/5.$ In this case we argue exactly as in the Laplacian case (see Theorem 4.1 in \cite{DFS}) using the key Lemma \ref{main}.


\qed

\subsection{Degenerate case}
In this case, the negative part of $u$ is negligible and the positive part is close to a one-plane solution (i.e. $\beta=0$).

\begin{thm}[Harnack inequality]\label{HIdg}There exists a universal constant $\bar
\ep$,  such that if $u$ satisfies at some point $x_0 \in B_2$

\be\label{oscdeg} U_0(x_n+ a_0) \leq u^+(x) \leq U_0(x_n+ b_0) \quad
\text{in $B_r(x_0) \subset B_2,$}\ee
with  $$\|u^-\|_{L^\infty} \leq \ep, \quad \|f\|_{L^\infty} \leq \ep^3$$ and
$$b_0 - a_0 \leq \ep r, $$ for some $\ep \leq \bar \ep,$ then
$$ U_{0}(x_n+ a_1) \leq u^+(x) \leq U_0(x_n+ b_1) \quad \text{in
$B_{r/20}(x_0)$},$$ with
$$a_0 \leq a_1 \leq b_1 \leq b_0, \quad
b_1 -  a_1\leq (1-c)\ep r, $$ and $0<c<1$ universal.
\end{thm}

From the theorem above we conclude the following.

\begin{cor} \label{corollary4}Let $u$ be as in Theorem $\ref{HIdg}$  satisfying \eqref{osc} for $r=1$. Then  in $B_1(x_0)$ $$\tilde
u_\ep:= \frac{u^+(x) - x_n}{\ep}$$ has a H\"older modulus of continuity at $x_0$, outside
the ball of radius $\ep/\bar \ep,$ i.e for all $x \in B_1(x_0)$, with $|x-x_0| \geq \ep/\bar\ep$
$$|\tilde u_\ep(x) - \tilde u_\ep (x_0)| \leq C |x-x_0|^\gamma.
$$
\end{cor}

Again, the proof of the Harnack inequality relies on the following lemma.

\begin{lem}\label{main2}There exists
a universal constant $\bar \ep>0$ such that if $u$ satisfies
\begin{equation*}  u^+(x) \geq U_0(x), \quad \text{in $B_1$}
\end{equation*}
with \be\label{dg} \|u^-\|_{L^\infty} \leq \ep^2, \quad \|f\|_{L^\infty} \leq \ep^4, \ee and at $\bar x=\dfrac{1}{5}e_n$ \be\label{u-p>ep2d}
u^+(\bar x) \geq U_0(\bar x_n + \ep), \ee then \be u^+(x) \geq
U_0(x_n+c\eps), \quad \text{in $\overline{B}_{1/2},$}\ee for some
$0<c<1$ universal. Analogously, if $$u^+(x) \leq U_0(x), \quad \text{in $B_1$}$$ and $$ u^+(\bar x) \leq U_0(\bar x_n - \eps),$$ then $$ u^+(x) \leq U_0(x_n - c \ep), \quad
\text{in $\overline{B}_{1/2}.$}$$
\end{lem}

\begin{proof} We prove the first statement. The proof follows the same line as in the non-degenerate case. The dependence of the Pucci operators on $\lambda/n,\Lambda$ is omitted.

Since $x_n >0$ in $B_{1/10}(\bar x)$  and $u^+ \geq U_0$ in $B_1$ we get
\begin{equation*} B_{1/10}(\bar x) \subset B_1^+(u). \end{equation*}

Thus $u-x_n \geq
0$ and solves $\mathcal{F}(D^2(u-x_n)) =f$  in $B_{1/10}(\bar x)$. We can apply Harnack inequality and the assumptions \eqref{dg} and  \eqref{u-p>ep2d}
to obtain that (for $\ep$ small enough)
\be\label{u-p>cep2} u
- x_n \geq  c_0\ep \quad \text{in $\overline B_{1/20}(\bar x)$}. \ee
Let $w$ be as in the proof of Lemma \ref{main} and  $\psi =1-w$. Set
\begin{equation*} v(x)= (x_n - \ep c_0 \psi(x))^+ - \ep^2 C_1 (x_n - \ep c_0 \psi(x))^-, \quad x \in \overline B_{3/4}(\bar x),\end{equation*} and for $t
\geq 0,$
$$v_t(x)= (x_n - \ep c_0 \psi +  t \ep)^+ -   \ep^2 C_1(x_n - \ep c_0 \psi(x) + t \ep)^-, \quad x \in \overline B_{3/4}(\bar x).
$$
Here $C_1$ is a universal constant to be made precise later.
We claim that
$$  v_0(x)=v(x) \leq u(x) \quad x \in \overline B_{3/4}(\bar x).$$

This is readily verified in the set where $u$ is non-negative using that $u \geq x_n^+.$  To prove our claim in the set where $u$ is negative we wish to use the following fact:
\be\label{negu}u^- \leq C x_n^- \ep^2, \quad \text{in $B_{\frac{19}{20}}$, $C$ universal}.\ee This estimate is  obtained remarking that in the set $\{u<0\},$
$u^-$ satisfies
$$
\mathcal{ M}^+(D^2u^-)= - \mathcal M^-(D^2u)\geq - \mathcal{F}(D^2u)= - f>-\eps^4.
$$
Hence,  the inequality follows using that $\{u<0\} \subset \{x_n <0\},$ $\|u^-\|_\infty < \ep^2$ and the comparison principle with the function $w$ satisfying $$\mathcal{ M}^+(D^2w) = -\eps^2 \leq -\eps^4 \quad \text{in $B_1 \cap \{x_n <0\}$},$$ $$w=\eps^2 \quad \text{on $\p B_1 \cap \{x_n <0\}$}, \quad w=0 \quad \text{on $x_n=0$.}$$ Notice that $\mathcal{M}^+$ is  a convex operator, thus $w/\eps^2$ is an explicit barrier which has $C^{2,\alpha}$ estimates up to $\{x_n=0\}$. Hence $u^- \leq w \leq C x_n^- \eps^2$ in $B_{19/20} \cap \{x_n \leq 0\}$.

Thus our claim immediately follows from the fact that for $x_n<0$ and $C_1 \geq C,$
$$\ep^2 C_1(x_n - \ep c_0 \psi(x)) \leq Cx_n \ep^2.$$

Let $\bar t$ be the largest $t \geq 0$ such that
$$v_{t}(x) \leq u(x) \quad \text{in $\overline B_{3/4}(\bar x)$}.$$

We want to show that $\bar t \geq c_0.$ Then we get the desired statement. Indeed, it is easy to check that if
$$u(x) \geq v_{\bar t}(x) = (x_n - \ep c_0 \psi + \bar t \ep)^+ -   \ep^2 C_1 (x_n - \ep c_0 \psi(x) + \bar t \ep)^-\quad \text{in $ B_{3/4}(\bar x)$}$$ then
$$u^+(x) \geq U_0(x_n + c \ep) \quad \text{in $ B_{1/2} \subset \subset B_{3/4}(\bar x)$}$$
with $c$ universal, $c < c_0 \inf_{B_1/2} w.$

Suppose $\bar t < c_0$. Then at some
$\tilde x \in \overline B_{3/4}(\bar x)$ we have $$v_{\bar t}(\tilde x) =
u(\tilde x).$$ We show that such touching point can only occur on $\overline B_{1/20}(\bar x).$
Indeed, since $w\equiv 0$ on $\p B_{3/4}(\bar x)$ from the definition of $v_t$ we get that for $\bar t < c_0$
$$v_{\bar t}(x) =(x_n -\ep c_0 + \bar t \ep)^+ - \ep^2 C_1 (x_n -\ep c_0 + \bar t \ep)^-< u(x)\quad \textrm{on  $\p B_{3/4}(\bar x)$}.$$

In the set where $u \geq 0$ this can be seen using that $u \geq x_n^+$ while in the set where $u<0$ again we can use the estimate \eqref{negu}.

We now show that $\tilde x$ cannot belong to the annulus $A$.
Indeed, $$\mathcal{F}(D^2 v_{\bar t})\geq \mathcal M^-(D^2 v_{\bar t}) \geq \ep^3c_0 k(n) > \ep^4 \geq \|f\|_{\infty}, \quad \textrm{in
$A^+(v_{\bar t}) \cup A^-(v_{\bar t})$}$$ for $\ep$ small enough.

Also,
$$(v_{\bar t}^+)_\nu^2 - (v_{\bar t}^-)_\nu^2 = (1 -\ep^4 C_1^2)(1 + \ep^2 c_0^2 |\nabla \psi|^2  - 2\ep c_0 \psi_n) \quad \text{on $F(v_{\bar t}) \cap A$}.$$

Thus,
$$(v_{\bar t}^+)_\nu^2 - (v_{\bar t}^-)_\nu^2 > 1  \quad \text{on $F(v_{\bar t}) \cap A$}$$
as long as $\ep$ is small enough (as in the non-degenerate case one can check that $\inf_ {F(v_{\bar t}) \cap A} (-\psi_n)>c > 0$, $c$ universal).
 Thus,
$v_{\bar t}$ is a strict subsolution to $\eqref{fb}$ in $A$ which lies below $u$, hence by definition $\tilde x$ cannot belong to $A.$

Therefore, $\tilde x \in \overline B_{1/20}(\bar x)$ and
$$u(\tilde x)=v_{\bar t}(\tilde x)  = (\tilde x_n + \bar t \ep)< \tilde x_n +  c_0 \ep$$ contradicting \eqref{u-p>cep2}.
\end{proof}

\section{Improvement of flatness}

In this section we prove our key ``improvement of flatness" lemmas. As in Section \ref{section4}, we need to distinguish two cases. Recall that $\mathcal E(\lambda,\Lambda)$ is the class of all uniformly elliptic operators $\mathcal F(M)$ with ellipticity constants $\lambda,\Lambda$ and such that $\mathcal F(0)=0.$

\subsection{Non-degenerate case}
In this case $u$ is trapped between two translations of a two-plane solution $U_\beta$ with $\beta \neq 0.$ We show that when we restrict to smaller balls, $u$ is trapped between closer translations of another two-plane solution (in a different system of coordinates).
\begin{lem}[Improvement of flatness] \label{improv1}Let $u$  satisfy
\begin{equation}\label{flat_1}U_\beta(x_n -\ep) \leq u(x) \leq U_\beta(x_n +
\ep) \quad \text{in $B_1,$} \quad 0\in F(u),
\end{equation} with $0 <  \beta \leq L$ and
$$\|f\|_{L^\infty(B_1)} \leq \ep^2\beta.$$

If $0<r \leq r_0$ for $r_0$ universal, and $0<\ep \leq \ep_0$ for some $\ep_0$
depending on $r$, then

\begin{equation}\label{improvedflat_2_new}U_{\beta'}(x \cdot \nu_1 -r\frac{\ep}{2})  \leq u(x) \leq
U_{\beta'}(x \cdot \nu_1 +r\frac{\ep }{2}) \quad \text{in $B_r,$}
\end{equation} with $|\nu_1|=1,$ $ |\nu_1 - e_n| \leq \tilde C\ep$ , and $|\beta -\beta'| \leq \tilde C\beta \ep$ for a
universal constant $\tilde C.$

\end{lem}

\begin{proof}We divide the proof of this Lemma into 3 steps.

 \

\textbf{Step 1 -- Compactness.} Fix $r \leq r_0$ with $r_0$ universal (the precise $r_0$ will be given in Step 3). Assume by contradiction that we
can find a sequence $\ep_k \rightarrow 0$ and a sequence $u_k$ of
solutions to \eqref{fb} in $B_1$ with  for a sequence of operators $\mathcal F_k \in \mathcal E(\lambda,\Lambda)$ and right hand sides $f_k$ with $L^\infty$ norm bounded by $\ep_k^2\beta_k$, such that
\begin{equation}\label{flat_k}U_{\beta_k}(x_n -\ep_k) \leq u_k(x) \leq U_{\beta_k}(x_n +
\ep_k) \quad \text{for $x \in B_1$,  $0 \in F(u_k),$}
\end{equation} with $L \geq \beta_k > 0$,
but $u_k$ does not satisfy the conclusion \eqref{improvedflat_2_new} of the lemma.

Set ($\alpha_k^2=1+\beta_k^2$),$$ \tilde{u}_{k}(x)= \begin{cases}\dfrac{u_k(x) - \alpha_k x_n}{\alpha_k \ep_k}, \quad x \in
B_1^+(u_k) \cup F(u_k) \\ \ \\ \dfrac{u_k(x) - \beta_k x_n}{\beta_k\ep_k}, \quad x \in
B_1^-(u_k).\end{cases}
$$
Then \eqref{flat_k} gives,
\begin{equation}\label{flat_tilde} -1 \leq \tilde{u}_{k}(x) \leq 1
\quad \text{for $x \in B_1$}.
\end{equation}

From Corollary \ref{corollary}, it follows that the function
$\tilde u_{k}$ satisfies \be\label{HC}|\tilde u_{k}(x) - \tilde
u_{k} (y)| \leq C |x-y|^\gamma,\ee for $C$ universal and
$$|x-y| \geq \ep_k/\bar\ep, \quad x,y \in B_{1/2}.$$ From \eqref{flat_k} it clearly follows that
$F(u_k)$ converges to $B_1 \cap \{x_n=0\}$ in the Hausdorff
distance. This fact and \eqref{HC} together with Ascoli-Arzela
give that as $\ep_k \rightarrow 0$ the graphs of the
$\tilde{u}_{k}$ converge (up to a
subsequence) in the Hausdorff distance to the graph of a H\"older
continuous function $\tilde{u}$ over $B_{1/2}$. Also, up to a subsequence
$$\beta_k \to \tilde \beta \geq 0$$ and hence $$\alpha_k \to \tilde \alpha = \sqrt{1+\tilde \beta^2}.$$

\

\textbf{Step 2 -- Limiting Solution.} We now show that $\tilde u$
solves
\begin{equation}\label{Neumann}
  \begin{cases}
    \mathcal{F}^\pm(D^2\tilde u(x))=0 & \text{in $B_{1/2}^\pm \cap \{x_n \neq 0\}$}, \\
\ \\
a (\tilde u_n)^+ - b (\tilde u_n)^-=0 & \text{on $B_{1/2} \cap \{x_n =0\}$},
  \end{cases}\end{equation}
  with $\mathcal F^\pm \in \mathcal{E}(\lambda,\Lambda)$, and $a=\tilde \alpha^2 >0, b=\tilde \beta^2 \geq 0.$

Set,
$$
\mathcal{F}_k^+(M)=\frac{1}{\alpha_k\epsilon_k}\mathcal{F}_k(\alpha_k\epsilon_k M)
$$
$$
\mathcal{F}_k^-(M)=\frac{1}{\beta_k\epsilon_k}\mathcal{F}_k(\beta_k\epsilon_k M).
$$
 Then $\mathcal{F}_k^\pm \in \mathcal{E}(\lambda,\Lambda).$ Thus, up to extracting a subsequence, $$\mathcal F_k^\pm \to \mathcal F^\pm, \quad \text{uniformly on compact subsets of matrices.}$$
Moreover,
  $$
  \mathcal{F}_k^+(D^2\tilde{u}_k(x))=  \frac{f_k}{\epsilon_k\alpha_k},\quad B_1^+(u_k)
  $$
  and
  $$
  \mathcal{F}_k^-(D^2\tilde{u}_k(x))=  \frac{f_k}{\epsilon_k\beta_k},\quad B_1^-(u_k),
  $$
 with  $$|\mathcal{F}_k^{\pm }(D^2 u_{k}(x))| \leq \ep_k$$  in $B_1^\pm(u_k),$
since $\|f\|_\infty \leq \ep_k^2\beta_k. $

Then, by standard arguments (see Proposition 2.9 in \cite{CC}),   we conclude that  $\tilde
u$ solves in the viscosity sense
$$\mathcal{F}^\pm(D^2\tilde{u})=0 \quad \text{in  $B_{1/2}^\pm(\tilde{u})$}.$$

Next, we prove that $\tilde u$ satisfies the boundary
condition in \eqref{Neumann} in the viscosity sense. By a slight modification of the argument after the Definition \ref{def_visc} (ii'), it is enough to test that if
$\tilde \phi$ is a function of the form ($\gamma$ a specific constant to be made precise later)
$$\tilde \phi(x) = A+ px_n^+- qx_n^- + B Q^\gamma(x-y)$$ with $$Q^\gamma(x) = \frac 1 2 [\gamma x_n^2 - |x'|^2], \quad y=(y',0), \quad A \in \R, B >0, $$ and $$a p- b q>0,$$ then $\tilde \phi$ cannot touch $u$ strictly by below at a point $x_0= (x_0', 0) \in B_{1/2}.$

The analogous statement by above follows with a similar argument.

Suppose that such a $\tilde \phi$ exists and let $x_0$ be the touching point.
Let \be\label{biggamma}\Gamma^\gamma(x) = \frac{1}{2\gamma} [(|x'|^2 + |x_n-1|^2)^{-\gamma} - 1],\ee
where $\gamma$ is sufficiently large (to be made precise later),
 and let

\be\label{biggammak}\Gamma_{k}^\gamma(x)=  \frac{1}{B\ep_k}\Gamma^\gamma(B\ep_k(x-y)+ A B \ep_k^2 e_n).\ee

Now,  call

$$\phi_k(x)= a_k \Gamma^{\gamma +}_k(x) - b_k \Gamma^{\gamma -}_k(x) + \alpha_k (d_k^+(x))^2\ep_k^{3/2} +\beta_k(d_k^-(x))^2\ep_k^{3/2}$$
where

$$a_k=\alpha_k(1+\ep_k p), \quad b_k=\beta_k(1+\ep_k q)$$ and $d_k(x)$ is the signed distance from $x$ to $\p B_{\frac{1}{B\ep_k}}(y+e_n(\frac{1}{B\ep_k}-A\ep_k)).$

Finally, let

$$ \tilde{\phi}_{k}(x)= \begin{cases}\dfrac{\phi_k(x) - \alpha_k x_n}{\alpha_k \ep_k}, \quad x \in
B_1^+(\phi_k) \cup F(\phi_k) \\ \ \\ \dfrac{\phi_k(x) - \beta_k x_n}{\beta_k\ep_k}, \quad x \in
B_1^-(\phi_k).\end{cases}
$$

By Taylor's theorem

$$ \Gamma(x)= x_n + Q^{\gamma}(x) + O(|x|^3) \quad x \in
B_1,$$
thus it is easy to verify that
$$ \Gamma_k^\gamma(x)= A\ep_k + x_n + B\ep_kQ^{\gamma}(x-y) + O(\ep_k^2) \quad x \in
B_1,$$
with the constant in $O(\ep_k^2)$ depending on $A,B,$ and $|y|$ (later this constant will depend also on $p,q$).

It follows that in $
B_1^+(\phi_k) \cup F(\phi_k) $ ($Q^{\gamma, y}(x)=Q^\gamma(x-y)$)

 $$ \tilde{\phi}_{k}(x)= A+BQ^{\gamma, y} + p x_n + A\ep_k p + Bp\ep_kQ^{\gamma, y} + \ep_k^{1/2}d_k^2+ O(\ep_k)
$$
and analogously in  $
B_1^-(\phi_k)$

$$ \tilde{\phi}_{k}(x)= A+BQ^{\gamma,y} + q x_n + A\ep_k p + Bq\ep_kQ^{\gamma,y} + \ep_k^{1/2}d^2_k + O(\ep_k).
$$

Hence, $\tilde \phi_k$ converges uniformly to $\tilde \phi$ on $B_{1/2}$. Since $\tilde u_k$ converges uniformly to $\tilde u$ and $\tilde \phi$ touches $\tilde u$ strictly by below at $x_0$, we conclude that there exist a sequence of constants $c_k \to 0$ and of points $x_k \to x_0$ such that the function
$$\psi_k(x) = \phi_k(x+\ep_k c_k e_n)$$
touches $u_k$ by below at $x_k$. We thus get a contradiction if we prove that $\psi_k$ is a strict subsolution to our free boundary problem, that is \begin{equation*}\label{fbpsi} \left \{
\begin{array}{ll}
   \mathcal{F}_k(D^2 \psi_k) > \ep_k^2\beta_k \geq \|f_k\|_{\infty},   & \hbox{in $B_1^+(\psi_k) \cup B_1^-(\psi_k) ,$}\\
   \ \\
(\psi_k^+)_\nu^2 - (\psi_k^-)^2_\nu >1, & \hbox{on $F(\psi_k)$.}\\
\end{array}\right.
\end{equation*}


 For $k$ large enough, say, in the positive phase of $\psi_k$ (denoting $\bar{x}=x+\varepsilon_k c_k e_n$ and dropping the dependance of the Pucci operator from $\lambda/n,\Lambda$ ), we have that
 $$\mathcal M^-(D^2 \psi_k(x)) \geq a_k \mathcal  M^-(D^2 \Gamma_k^\gamma(\bar x)) + \alpha_k \eps_k^{3/2} \mathcal M^-(d_k^2(\bar x)).$$

 As computed several times throughout the paper (see for example Lemma \ref{main}), for $\gamma$ large enough depending on $n,\lambda,\Lambda$ we have that $ \mathcal  M^-(D^2 \Gamma_k^\gamma(\bar x))>0.$
 Moreover, in the appropriate system of coordinates,
 $$D^2 d_k^2 (\bar x) = diag\{-d_k(\bar x) \kappa_1(\bar x),\ldots, -d_k(\bar x) \kappa_{n-1}(\bar x), 1\}$$
 where the
 $\kappa_i(\bar x)$ denote the curvature of the surface parallel to $\p B_{\frac{1}{B\eps_k}}(y+ e_n (\frac{1}{B\eps_k}-A\eps_k))$ which passes through $\bar x.$ Thus,

 $$\kappa_i (\bar x) =\frac{B \eps_k}{1-B \eps_kd_k(\bar x)}.$$

 For $k$ large enough we conclude that $ \mathcal M^-(d_k^2(\bar x))> \lambda/2n$ and hence,
 $$\mathcal F_k(D^2 \psi_k) \geq \mathcal M^-(D^2 \psi_k(x)) \geq \alpha_k \eps_k^{3/2} \frac{\lambda}{2n}  > \beta_k \eps_k^2 \geq \|f_k\|_{\infty},$$ as desired.

An analogous estimate holds in the
negative phase.

Finally, since on the zero level set $|\nabla \Gamma_k|=1$ and $|\nabla
d^2_k|=0,$ the free boundary condition reduces to show that
$$
a_k^2-b_k^2>1.
$$

Recalling the definition of $a_k,b_k$ we need to check that
$$
(a_k^2p^2-\beta_k^2q^2)\varepsilon+2(\alpha_k^2p-\beta_k^2q)>0.
$$

This inequality holds for $k$ large since
$$
\tilde{\alpha}^2p-\tilde{\beta}^2q=a p -b q>0.
$$
%
%

%
%
%
%
%

Thus $\tilde u$ is a solution to the linearized problem.

\

\textbf{Step 3 -- Contradiction.} The conclusion now follows exactly as in the case of \cite{DFS}, using the regularity estimates for the solution of the transmission problem from Theorem \ref{lineareg}.
\end{proof}

\subsection{Degenerate case}In this case, the negative part of $u$ is negligible and the positive part is close to a one-plane solution (i.e. $\beta=0$). We prove below that in this setting only $u^+$ enjoys an improvement of flatness.

\begin{lem}[Improvement of flatness] \label{improv4}Let $u$  satisfy
\begin{equation}\label{flat}U_0(x_n -\ep) \leq u^+(x) \leq U_0(x_n +
\ep) \quad \text{in $B_1,$} \quad 0\in F(u),
\end{equation} with $$\|f\|_{L^\infty(B_1)} \leq \ep^4,$$ and $$ \|u^-\|_{L^\infty(B_1)} \leq \ep^2.$$

If $0<r \leq r_1$ for $r_1$ universal, and $0<\ep \leq \ep_1$ for some $\ep_1$
depending on $r$, then

\begin{equation}\label{improvedflat_2}U_0(x \cdot \nu_1 -r\frac{\ep}{2})  \leq u^+(x) \leq
U_0(x \cdot \nu_1 +r\frac{\ep }{2}) \quad \text{in $ B_r,$}
\end{equation} with $|\nu_1|=1,$ $ |\nu_1 - e_n| \leq C\ep$  for a
universal constant $C.$

\end{lem}

\begin{proof}We argue similarly as in the non-degenerate case.

 \

\textbf{Step 1 -- Compactness.} Fix $r \leq r_0$ with $r_0$ universal (the precise $r_0$ will be given in Step 3). Assume by contradiction that we
can find a sequence $\ep_k \rightarrow 0$ and a sequence $u_k$ of
solutions to \eqref{fb} in $B_1$ for operators $\mathcal F_k \in \mathcal E(\lambda,\Lambda)$ and right hand sides $f_k$ with $L^\infty$ norm bounded by $\ep_k^4$, such that
\begin{equation}\label{flat_k4}U_{0}(x_n -\ep_k) \leq u_k(x) \leq U_{0}(x_n +
\ep_k) \quad \text{for $x \in B_1$,  $0 \in F(u_k),$}
\end{equation} with $$\|u_k^-\|_\infty \leq \ep_k^2$$
but $u_k$ does not satisfy the conclusion \eqref{improvedflat_2} of the lemma.

Set $$ \tilde{u}_{k}(x)= \dfrac{u_k(x) -  x_n}{ \ep_k}, \quad x \in
B_1^+(u_k) \cup F(u_k) $$
Then \eqref{flat_k} gives,
\begin{equation*}\label{flat_tilde} -1 \leq \tilde{u}_{k}(x) \leq 1
\quad \text{for $x \in
B_1^+(u_k) \cup F(u_k) $}.
\end{equation*}

As in the previous case, it follows from Corollary \ref{corollary4} that as $\ep_k \rightarrow 0$ the graphs of the
$\tilde{u}_{k}$ converge (up to a
subsequence) in the Hausdorff distance to the graph of a H\"older
continuous function $\tilde{u}$ over $B_{1/2} \cap \{x_n \geq 0\}$.

\

\textbf{Step 2 -- Limiting Solution.} We now show that $\tilde u$ solves the following Neumann problem
\begin{equation}\label{Neumann4}
  \begin{cases}
    \mathcal{ F}_*^+(D^2\tilde u)=0 & \text{in $B_{1/2} \cap \{x_n > 0\}$}, \\
\ \\
\tilde u_n =0 & \text{on $B_{1/2} \cap \{x_n =0\}$},
  \end{cases}\end{equation}
  with $\mathcal{F}_*^+ \in \mathcal E(\lambda,\Lambda).$
 As before, the interior condition follows easily, thus we focus on the boundary condition.
It is enough to test that if
$\tilde \phi$ is a function of the form ($\gamma$ a precise constant to be specified later)
$$\tilde \phi(x) = A+ px_n + B Q(x-y)$$ with $$Q^\gamma(x) = \frac 1 2 [\gamma x_n^2 - |x'|^2], \quad y=(y',0), \quad A \in \R, B>0,\eta >0$$ and $$p >0,$$  then  $\tilde \phi$ cannot touch $\tilde u$ strictly by below at a point $x_0= (x_0', 0) \in B_{1/2}$.
Suppose that such a $\tilde \phi$ exists and let $x_0$ be the touching point.

Let  $\Gamma_k^\gamma$ be as in the proof of the non-degenerate case (see \eqref{biggammak}). Call
$$\phi_k(x)= a_k \Gamma^{\gamma +}_k(x)+ (d_k^+(x))^2\ep_k^2, \quad a_k=(1+\ep_k p) $$
where
$d_k(x)$ is the signed distance from $x$ to $\p B_{\frac{1}{B\ep_k}}(y+e_n(\frac{1}{B\ep_k}-A\ep_k)).$

Let

$$ \tilde{\phi}_{k}(x)=\dfrac{\phi_k(x) -  x_n}{ \ep_k}. $$

As in the previous case, it follows that in $
B_1^+(\phi_k) \cup F(\phi_k) $ ($Q^{\gamma,y}(x)=Q^\gamma(x-y)$)

 $$ \tilde{\phi}_{k}(x)= A+BQ^{\gamma,y} + p x_n + A\ep_k p + Bp\ep_kQ^{\gamma, y} + \ep_kd_k^2+ O(\ep_k).
$$

Hence, $\tilde \phi_k$ converges uniformly to $\tilde \phi$ on $B_{1/2} \cap \{x_n \geq 0\}$. Since $\tilde u_k$ converges uniformly to $\tilde u$ and $\tilde \phi$ touches $\tilde u$ strictly by below at $x_0$, we conclude that there exist a sequence of constants $c_k \to 0$ and of points $x_k \to x_0$ such that the function
$$\psi_k(x) = \phi_k(x+\ep_k c_k e_n)$$
touches $u_k$ by below at $x_k \in B_1^+(u_k) \cup F(u_k)$. We claim that $x_k$ cannot belong to $B_1^+(u_k)$. Otherwise, in a small neighborhood $N$ of $x_k$  we would have that  (with a similar computation as in the non-degenerate case, and $\gamma$ large enough universal)
$$  \mathcal{M}^-(D^2\psi_k) > \ep_k^4 \geq \|f_k\|_{\infty} \geq  \mathcal{F}(D^2 u_k)\geq \mathcal{M}^-(D^2u_k), $$
$\psi_k < u_k$ in $N\setminus\{x_k\}, \psi_k(x_k) = u_k(x_k),$
a contradiction.

Thus $x_k \in F(u_k) \cap \p B_{\frac{1}{B\ep_k}}(y+e_n(\frac{1}{B\ep_k}-A\ep_k-\eps_kc_k)).$ For simplicity we call $$\mathcal B: =B_{\frac{1}{B\ep_k}}(y+e_n(\frac{1}{B\ep_k}-A\ep_k-\eps_kc_k)).$$ Let $\mathcal N$ be a neighborhood of $x_k$.
In the set $\{u_k <0\},$ $$\mathcal M^+(D^2 u^-_k) = -\mathcal M^-(D^2 u_k) \geq F(D^2 u_k) = f \geq -\eps^4.$$

Hence, since $\|u^-_k\|_\infty \leq \eps^2_k$, we can compare $u_k^-$ with the function $\eps_k^2 w$ where $w$ solves the following problem:
$$\mathcal{ M}^+(D^2w) = -1  \quad \text{in $\mathcal N \setminus \overline{\mathcal B}$},$$ $$w=1 \quad \text{on $\p N \setminus \mathcal B $}, \quad w=0 \quad \text{on $\mathcal N \cap \p \mathcal B$}.$$

Let \begin{equation}\label{Psi}\Psi_k(x) = \begin{cases}\psi_k & \text{in $\mathcal N \cap \mathcal B$}\\  \ & \ \\  -\eps_k^2 w & \text{in $\mathcal N \setminus \mathcal B.$}\end{cases}\ee

Then  $\Psi_k$ touches $u_k$ strictly by below at $x_k \in F(u_k) \cap F(\Psi_k)$.
We reach a contradiction if we show that
$$(\Psi_k^+)_\nu^2 - (\Psi_k^-)^2_\nu >1, \quad \hbox{on $F(\Psi_k)$.}$$
This is equivalent to showing that (for $c$ small universal)
$$a_k^2 - c \eps_k^4 >1$$ or
$$(1+\ep_k p)^2 -c \ep_k^4> 1.$$
This holds for $k$ large enough since $p>0,$ and our proof is concluded.

\

\textbf{Step 3 -- Contradiction.} In this step we can argue as in the final step of the proof of Lemma 4.1 in [D]. \end{proof}

\section{Proof of the main Theorem.}

In this section we exhibit the proofs of our main results, Theorem \ref{flatmain1} and Theorem \ref{Lipmain}.
We recall the following elementary lemma from \cite{DFS} which holds for any continuous function $u.$

\begin{lem}\label{elementary} Let $u$ be a continuous function. If for $\eta>0$ small, $$\|u - U_{\beta}\|_{L^{\infty}(B_{2})} \leq \eta, \quad \text{$0 \leq \beta \leq L,$}$$ and
\begin{equation*} \{x_n \leq - \eta\} \subset B_2 \cap \{u^+(x)=0\} \subset \{x_n \leq \eta \},\end{equation*} then \begin{itemize}\item If $\beta \geq \eta^{1/3}$, then $$U_\beta(x_n - \eta^{1/3}) \leq u(x) \leq U_\beta(x_n + \eta^{1/3}),\quad \text{in $B_1$}$$ \\
\item If $\beta < \eta^{1/3},$ then $$U_0(x_n - \eta^{1/3}) \leq u^+(x) \leq U_0(x_n + \eta^{1/3}),\quad \text{in $B_1$}.$$
\end{itemize}\end{lem}

\subsection{Proof of Theorem \ref{flatmain1}.}
To complete the analysis of the degenerate case, we need to deal with the situation when $u$ is close to a one-plane solution and however the size of $u^-$ is not negligible. Precisely, we prove the following lemma.

\begin{lem}\label{finalcase}
Let $u$  solve \eqref{fb} in $B_2$ with $$\|f\|_{L^\infty(B_1)} \leq \ep^4$$ and satisfy
\begin{equation}\label{flat}U_0(x_n -\ep) \leq u^+(x) \leq U_0(x_n +
\ep) \quad \text{in $B_1,$} \quad 0\in F(u),
\end{equation}$$\|u^-\|_{L^{\infty}(B_2)} \leq \bar C\eps^2,  \quad \|u^-\|_{L^\infty(B_1)} > \ep^2,$$ for a universal constant $\bar C.$ If $\ep \leq \ep'$ universal,  then the rescaling
$$u_\ep(x) = \ep^{-1/2}u(\ep^{1/2}x)$$ satisfies in $B_1$
$$U_{\beta'}(x_n - C'\ep^{1/2}) \leq u_{\ep}(x) \leq U_{\beta'}(x_n + C'\ep^{1/2})$$ with $\beta' \sim \eps^2$ and $C'>0$ depending on $\bar C$.
\end{lem}
\begin{proof} As usually, we omit the dependence of the Pucci operators from $\lambda/n, \Lambda$. For notational simplicity we set $$v = \frac{u^-}{\ep^2}.$$

From our assumptions we can deduce that $$F(v) \subset \{-\ep \leq x_n \leq \ep\}$$ \be\label{neg}v \geq 0 \quad \text{in $B_2 \cap \{x_n \leq -\ep\}$}, \quad v \equiv 0 \quad \text{in  $B_2 \cap \{x_n > \ep\}.$}\ee

Also,
$$
\mathcal{M}^+(D^2v)= \frac{1}{\eps^2}\mathcal M^+(-D^2u)= -\frac{1}{\eps^2} \mathcal M^-(D^2u) \geq \frac{1}{\eps^2} \mathcal F(D^2u)\geq-\epsilon^2
$$
in $B_2 \cap \{x_n < -\ep\},$ and \be\label{C}0 \leq v \leq C \quad \text{on $\p B_2$,}\ee \be\label{barx}v(\bar x) >1 \quad \text{at some point $\bar x$ in $B_1.$}\ee

Hence,  using comparison with the function $w$ such that $$\mathcal{ M}^+(D^2w) = -1  \quad \text{in $D:= B_2 \cap \{x_n <\eps\}$},$$ $$w=C \quad \text{on $\p B_2 \cap \{x_n <\eps\}$}, \quad w=0 \quad \text{on $\{x_n =\eps\}$}$$ ($\mathcal{M}^+$ is  a convex operator hence $w$ is an explicit barrier which has $C^{1,1}$ estimates up to $\{x_n=\eps\}$,)
we get that for some $k>0$ universal
\be\label{ke}v \leq k|x_n -\ep|, \quad \text{in $B_1$}.\ee  This fact forces the point $\bar x$ in \eqref{barx} to belong to $B_1\cap \{x_n < -\ep\}$ at a fixed distance $\delta$ from $x_n = -\ep.$

Analogously,
$$
\mathcal{M}^-(D^2v)= \frac{1}{\eps^2}\mathcal M^-(-D^2u)= -\frac{1}{\eps^2} \mathcal M^+(D^2u) \leq \frac{-1}{\eps^2} \mathcal F(D^2u)\leq \epsilon^2
$$
in $B_2 \cap \{x_n < -\ep\}.$ Thus if $w$ is such that $\mathcal{M}^-(D^2w)=0$ in $B_1 \cap \{x_n < -\ep\}$ such that $$w=0 \quad \text{on $B_1 \cap \{x_n=-\ep\}$}, \quad w=v \quad \text{on $\p B_1 \cap  \{x_n \leq  -\ep\}$},$$
then
$$\mathcal{M}^-(D^2(w+ \frac{1}{2\lambda}\ep^2(|x|^2-3))\geq  \ep^2.
$$
By the comparison  principle we conclude that
\be w+ \frac{1}{2\lambda}\ep^2(|x|^2-3) \leq v \quad \mbox{on}\quad B_1 \cap \{x_n < -\ep\}.\ee
Also, for $\ep$ small, in view of \eqref{ke} we obtain that
\be  w- k\ep(|x|^2-3) \geq v \quad \mbox{on}\quad \p (B_1 \cap \{x_n < -\ep\})\ee and hence also in the interior. Thus we conclude that
\be\label{w-v} |w-v| \leq c\eps \quad \text{in $B_1 \cap \{x_n < -\ep\}$}. \ee In particular this is true at $\bar x$ which forces \be\label{wbarx} w(\bar x) \geq 1/2.\ee By expanding $w$ around $(0, -\ep)$  we then obtain, say in $B_{1/2} \cap  \{x_n \leq -\ep\}$
$$|w - a |x_n + \ep|| \leq C |x|^2 +C\ep.$$
This combined with \eqref{w-v} gives that
\be |v - a|x_n+\ep|| \leq C\ep, \quad \text{in $B_{\ep^{1/2}} \cap  \{x_n \leq -\ep\}$.}\ee
Moreover, in view of \eqref{wbarx} and the fact that $\bar x$ occurs at a fixed distance from $\{x_n = -\ep\}$ we deduce from Hopf lemma that
$$a \geq c>0$$ with $c$ universal. In conclusion (see \eqref{ke})
\be\label{u-}  |u^- - b\ep|x_n+\ep|| \leq C\ep^3, \quad \text{in $B_{\ep^{1/2}} \cap  \{x_n \leq -\ep\}$,} \quad u^- \leq b\ep^2 |x_n-\ep|, \quad \text{in $B_1$}\ee with $b$ comparable to a universal constant.

Combining \eqref{u-} and the assumption \eqref{flat} we conclude that in $B_{\ep^{1/2}}$
\be (x_n - \ep)^+ - b\ep(x_n-C\ep)^- \leq u(x) \leq (x_n+\ep)^+ - b \ep (x_n+C\ep)^-\ee
with $C>0$ universal and $b$ larger than a universal constant. Rescaling, we obtain that in $B_1$
\be (x_n - \ep^{1/2})^+ - \beta'(x_n-C\ep^{1/2})^- \leq u_\ep(x) \leq (x_n+\ep^{1/2})^+ - \beta' (x_n+C\ep^{1/2})^-\ee
with $\beta' \sim  \ep^2$. We finally need to check that this implies the desired conclusion in $B_1$
$$ \alpha'(x_n - C\ep^{1/2})^+ - \beta'(x_n-C\ep^{1/2})^- \leq u_\ep(x) \leq \alpha'(x_n+C\ep^{1/2})^+ - \beta' (x_n+C\ep^{1/2})^-$$ with $\alpha'^2=1+\beta^2 \sim 1+ \ep^4.$ This clearly holds in $B_1$ for $\ep$ small, say  by possibly enlarging $C$ so that $C \geq 2.$

\end{proof}

We are finally ready to exhibit the proof of our main Theorem \ref{main_new}. Having provided all the necessary ingredients, the proof now follows as in \cite{DFS}. For the reader's convenience we present the details.

\

\textit{Proof of Theorem \ref{main_new}.} Let us fix $\bar r >0$ to be a universal constant such that
$$\bar r \leq r_0, r_1, 1/8,$$
with $r_0,r_1$ the universal constants in the improvement of flatness Lemmas \ref{improv1}-\ref{improv4}. Also, let us fix a universal constant $\tilde \ep>0$ such that
$$\tilde \ep \leq \ep_0(\bar r), \frac{\ep_1(\bar r)}{2}, \frac{1}{2\tilde C}, \frac{\ep_0(\bar r)^2}{2C'}, \frac{\ep'}{4}, C''$$ with $\ep_0,\ep_1,\ep', \tilde C, C', \bar C, $ the constants in the Lemmas \ref{improv1}-\ref{improv4}-\ref{finalcase} and $C''$ universal to be specified later.

Now, let $$\bar \ep = \tilde \ep^3.$$ We distinguish two cases. For notational simplicity we assume that $u$ satisfies our assumptions in the ball $B_2$ and $0 \in F(u)$.

\

\textit{Case 1.} $\beta \geq \tilde \ep.$

\

In this case, in view of Lemma \ref{elementary}and our choice of $\tilde \ep$, we obtain that $u$ satisfies the assumptions of Lemma \ref{improv1}, \begin{equation*}\label{flat*}U_\beta(x_n -\tilde \ep) \leq u(x) \leq U_\beta(x_n +
\tilde \ep) \quad \text{in $B_1,$} \quad 0\in F(u),
\end{equation*} with $0 <  \beta \leq L$ and
$$\|f\|_{L^\infty(B_1)} \leq \tilde \ep^3 \leq \tilde \ep^2\beta.$$ Thus we can conclude that, ($\beta_1=\beta'$)
\begin{equation}U_{\beta_1}(x \cdot \nu_1 -\bar r\frac{\tilde \ep}{2})  \leq u(x) \leq
U_{\beta_1}(x \cdot \nu_1 +\bar r\frac{\tilde \ep }{2}) \quad \text{in $B_{\bar r},$}
\end{equation} with $|\nu_1|=1,$ $ |\nu_1 - e_n| \leq \tilde C\tilde \ep$ , and $|\beta -\beta_1| \leq \tilde C\beta \tilde \ep$. In particular, by our choice of $\tilde \ep$ we have
$$\beta_1 \geq \tilde \ep/2.$$We can therefore rescale and iterate the argument above. Precisely, set ($k=0,1,2....$)
$$\rho_k = \bar r^k, \quad \ep_k = 2^{-k}\tilde \ep$$
and $$\mathcal F_k(M)= \rho_k \mathcal F(\frac{1}{\rho_k} M), \quad u_k(x) = \frac{1}{\rho_k} u(\rho_k x), \quad f_k(x) = \rho_k f(\rho_k x).$$ Notice that $F_k \in \mathcal E(\lambda,\Lambda)$ hence our flatness theorem holds.

Also, let $\beta_k$ be the constants generates at each $k$-iteration, hence satisfying ($\beta_0=\beta$)
$$|\beta_k-\beta_{k+1}| \leq \tilde C \beta_k \ep_k.$$
Then  we obtain by induction that each $u_k$ satisfies
\begin{equation}U_{\beta_k}(x \cdot \nu_k -\ep_k)  \leq u_k(x) \leq
U_{\beta_k}(x \cdot \nu_k +\ep_k ) \quad \text{in $B_1,$}
\end{equation} with $|\nu_k|=1,$ $ |\nu_k - \nu_{k+1}| \leq \tilde C\tilde \ep_k$ ($\nu_0=e_n$.)

\

\textit{Case 2.}  $\beta < \tilde \ep.$

\

In view of Lemma \ref{elementary} we conclude that
\begin{equation}U_0(x_n -\tilde \ep) \leq u^+(x) \leq U_0(x_n +
\tilde \ep) \quad \text{in $B_1.$}
\end{equation}  Moreover, from the assumption \eqref{initialass} and the fact that $\beta < \tilde \ep$ we also obtain that
$$\|u^-\|_{L^\infty (B_1)} < 2\tilde \ep.$$
Call $\ep' = 2\tilde \ep.$ Then $u$ satisfies the assumptions of the (degenerate) improvement of flatness Lemma \ref{improv4}.
\begin{equation}U_0(x_n -\ep') \leq u^+(x) \leq U_0(x_n +
\ep') \quad \text{in $B_1,$}
\end{equation} with $$\|f\|_{L^{\infty}(B_1)} \leq (\ep')^3, \quad \|u^-\|_{L^\infty(B_1)} < \ep'.$$
We conclude that
\begin{equation}U_0(x \cdot \nu_1 -\bar r\frac{\ep}{2})  \leq u^+(x) \leq
U_0(x \cdot \nu_1 +\bar r\frac{\ep }{2}) \quad \text{in $ B_{\bar r},$}
\end{equation} with $|\nu_1|=1,$ $ |\nu_1 - e_n| \leq C\ep'$  for a
universal constant $C.$
We now rescale as in the previous case and set ($k=0,1,2....$)
$$\rho_k = \bar r^k, \quad \ep_k = 2^{-k}\ep'$$
and $$\mathcal F_k(M)= \rho_k \mathcal F(\frac{1}{\rho_k} M),  \quad u_k(x) = \frac{1}{\rho_k} u(\rho_k x), \quad f_k(x) = \rho_k f(\rho_k x).$$
We can iterate our argument and obtain that (with $|\nu_k|=1,$ $ |\nu_k - \nu_{k+1}| \leq C\ep_k$)
\begin{equation}U_0(x \cdot \nu_k -\ep_k)  \leq u_k^+(x) \leq
U_0(x \cdot \nu_k +\ep_k) \quad \text{in $ B_{1},$}
\end{equation}
as long as we can verify that
$$\|u_k^-\|_{L^\infty(B_1)} < \ep_k^2.$$
Let $\bar k$ be the first integer $\bar k\geq 1$ for which this fails, that is
$$\|u_{\bar k}^-\|_{L^\infty(B_1)} \geq \ep_{\bar k}^2,$$ and $$\|u_{\bar k-1}\|_{L^\infty(B_1)} < \ep_{\bar k-1}^2.$$
Also, $$U_0(x \cdot \nu_{\bar k-1} -\ep_{\bar k-1})  \leq u_{\bar k-1}^+(x) \leq
U_0(x \cdot \nu_{\bar k-1} +\ep_{\bar k_-1}) \quad \text{in $ B_{1}.$}
$$
As argued several times, we can then conclude from the comparison principle that
$$u^-_{\bar k-1} \leq M|x_n-\ep_{\bar k-1}|\ep^2_{\bar k-1}\quad \text{in $ B_{19/20},$}
$$ for a universal constant $M>0. $ Thus, by rescaling we get that
$$\|u^-_{\bar k}\|_{L^\infty(B_2)} < \bar C \ep_{\bar k}^2$$ with $\bar C$ universal depending on the fixed $\bar r$. We obtain that $u_{\bar k}$ satisfies all the assumptions of that Lemma and hence  the rescaling
$$v(x) = \ep_{\bar k}^{-1/2}u_{\bar k}(\ep_{\bar k}^{1/2}x)$$ satisfies in $B_1$
$$U_{\beta'}(x_n - C'\ep_{\bar k}^{1/2}) \leq v(x) \leq U_{\beta'}(x_n + \bar C'\ep_{\bar k}^{1/2})$$ with $\beta' \sim \eps^2_{\bar k}.$
Call $\eta=\bar C \ep_{\bar k}^{1/2}.$
Then $v$ satisfies our free boundary problem in $B_1$ for the operator
$$\mathcal G(M)= \eps_{\bar k}^{1/2} \mathcal F_{\bar k}(\frac{1}{\eps_{\bar k}^{1/2}} M) \in \mathcal E(\lambda,\Lambda)$$
 with right hand side $$g(x)= \ep_{\bar k}^{1/2} f_{\bar k}(\ep_{\bar k}^{1/2}x)$$ and the flatness assumption
$$U_{\beta'}(x_n -\eta) \leq v(x) \leq U_{\beta'}(x_n + \eta)$$
Since $\beta' \sim \ep_{\bar k}^2$ with a universal constant, $$\|g\|_{L^\infty(B_1)} \leq \ep_{\bar k}^{1/2} \ep_{\bar k}^{4} \leq \eta^2 \beta'$$
as long as $\tilde \ep \leq C''$ universal depending on $\bar C$. In conclusion $v$ falls under the assumptions of the (non-degenerate) improvement of flatness Lemma \ref{improv1} and we can use an iteration argument as in Case 1.

\qed

   \subsection{Proof of Theorem \ref{Lipmain}.} To provide the proof of Theorem \ref{Lipmain}, we use the following Liouville type result for global viscosity solutions to a two-phase homogeneous free boundary problem, that could be of independent interest.

\begin{lem}\label{Liouville} Let $U$ be a global viscosity solution to \begin{equation}\label{fbglobal} \left \{
\begin{array}{ll}
   \mathcal G(D^2U) = 0,   & \hbox{in $\{U>0\}\cup \{U \leq 0\}^0,$}\\
\ \\
(U_\nu^+)^2 - (U_\nu^-)^2= 1, & \hbox{on $F(U):= \partial \{U>0\}.$} \\
\end{array}\right.
\end{equation}
Assume that $\mathcal G \in \mathcal E(\lambda,\Lambda)$ and  $\mathcal G$ is homogeneous of degree 1. Also, $F(U) = \{x_n =g(x'), x' \in \R^{n-1}\}$ with $Lip(g) \leq M$.  Then $g$ is linear and $U(x)=U_\beta(x)$ for some $\beta \geq 0.$
\end{lem}
\begin{proof}Assume for simplicity, $0 \in F(U)$. Also, balls (of radius $\rho$ and centered at $0$) in $\R^{n-1}$ are denoted by $\mathcal B_\rho.$

By the regularity theory in \cite{F1} , since $U$ is a solution in $B_2$, the free boundary $F(U)$ is $C^{1,\gamma}$ in $B_1$ with a bound depending only on $n,\lambda,\Lambda$ and $M$. Thus,
$$|g(x') - g(0) -\nabla g(0) \cdot x'| \leq C |x'|^{1+\alpha}, \quad x' \in \mathcal B_1 $$ with $C$ depending only on $n, \lambda, \Lambda, M.$
Moreover, since $U$ is a global solution, the rescaling
$$g_R(x') = \frac{1}{R} g(Rx'), \quad x' \in \mathcal B_2$$
which preserves the same Lipschitz constant as $g$, satisfies the same inequality as above i.e.
$$|g_R(x') - g_R(0) -\nabla g_R(0) \cdot x'| \leq C |x'|^{1+\alpha}, \quad x' \in \mathcal B_1.$$
This reads,
$$|g(Rx') - g(0) -\nabla g(0) \cdot Rx'| \leq C R |x'|^{1+\alpha}, \quad x' \in \mathcal B_1.$$
Thus,
$$|g(y') - g(0) -\nabla g(0) \cdot y'| \leq C \frac{1}{R^\alpha}|y'|^{1+\alpha}, \quad y' \in \mathcal B_R.$$

Passing to the limit as $R \to \infty$ we obtain the desired claim.
\end{proof}

\

Now the proof of Theorem \ref{Lipmain}, follows exactly as in the Laplacian case \cite{DFS}.

\textit{Proof of Theorem $\ref{Lipmain}.$} Let $\bar \eps$ be the universal constant in Theorem \ref{main_new}. Consider the blow-up sequence $$u_k(x) = \frac{u(\delta_k x)}{\delta_k}$$ with $\delta_k \to 0$ as $k \to \infty$. Each $u_k$ solves \eqref{fb} with operator $\mathcal F_k$ and right hand side $f_k$ given by $$\mathcal F_k(M)= \delta_k F_k(\frac{1}{\delta_k} M) \in \mathcal E(\lambda, \Lambda), \quad f_k(x) = \delta_k f(\delta_k x)$$ and $$\|f_k (x)\| \leq  \delta_k \|f\|_{L^\infty} \leq \bar \eps$$ for $k$ large enough. Standard arguments (see for example \cite{ACF}) using the uniform Lischitz continuity of the $u_k$'s and the nondegeneracy of their positive part $u_k^+$ (see  Lemma \ref{deltand}) imply that (up to a subsequence)
$$u_k \to \tilde u \quad \text{uniformly on compacts}$$
and
$$\{u_k^+=0\} \to \{\tilde u = 0\} \quad \text{in the Hausdorff distance}.$$

Moreover, up to a subsequence, the $\mathcal F_k$ converge uniformly on compact subsets of matrices to an operator $\tilde{\mathcal F} \in \mathcal E(\lambda, \Lambda).$ Since all the $\mathcal F_k$'s are homogeneous of degree 1, also $\tilde{\mathcal F}$ is homogeneous of degree 1.
The blow-up limit $\tilde u$ solves the global two-phase free boundary problem
\begin{equation}\label{fbglobal*} \left \{
\begin{array}{ll}
    \tilde{\mathcal F} (D^2 \tilde u) = 0,   & \hbox{in $\{\tilde u>0\} \cup \{\tilde u \leq 0\}^0$} \\
\ \\
 (\tilde u_\nu^+)^2 - (\tilde u_\nu^-)^2= 1, & \hbox{on $F(\tilde u):= \p \{\tilde u>0\}.$}  \\
\end{array}\right.
\end{equation}

Since $F(u)$ is a Lipschitz graph in  a neighborhood of 0, it follows from Lemma \ref{Liouville} that $\tilde u$ is a two-plane solutions, $\tilde u= U_\beta$ for some $\beta \geq 0$. Thus, for $k$ large enough
$$\|u_k - U_\beta\|_{L^\infty} \leq \bar \ep$$
and
\begin{equation*} \{x_n \leq - \bar \ep\} \subset B_1 \cap \{u^+_k(x)=0\} \subset \{x_n \leq \bar \ep \}.\end{equation*}  Therefore, we can apply our flatness Theorem \ref{main_new} and conclude that $F(u_k)$ and hence $F(u)$ is smooth.

\end{document}